\documentclass[11pt,letterpaper]{amsart}

\usepackage{mathrsfs,amsfonts,amsmath,amssymb}
\usepackage{latexsym}
\usepackage{comment} 

\usepackage{footnote} 
\usepackage[symbol*]{footmisc}

\newtheorem{satz}{Theorem}
\newtheorem{proposition}[satz]{Proposition}
\newtheorem{theorem}[satz]{Theorem}
\newtheorem{lemma}[satz]{Lemma}
\newtheorem{claim}[satz]{Claim}
\newtheorem{definition}[satz]{Definition}
\newtheorem{corollary}[satz]{Corollary}
\newtheorem{remark}[satz]{Remark}

\newtheorem{problem}[satz]{Problem}
\newtheorem{assumption}[satz]{Assumption}

\def\Z{\mathbb {Z}}
\def\F{\mathbb {F}}
\def\E{\mathsf{E}}

\def\a{\alpha}

\def\d{\delta}
\def\o{\omega}
\def\({\big (}
\def\){\big )}

\def\G{\Gamma}

\def\le{\leqslant}
\def\ge{\geqslant}
\def\_phi{\varphi}
\def\eps{\varepsilon}

\def\Gr{{\mathbf G}}
\def\FF{\widehat}

\def\ov{\overline}

\def\t{\tilde}

\def\la{\lambda}
\def\D{\Delta}

\def\SL{{\rm SL}}

\newcommand{\zk}{\mathcal{M}}

\newcommand{\bp}{\bigskip}

\author{I.D. Shkredov}
\title{On some results of Korobov and Larcher and Zaremba's conjecture}
\date{}
\begin{document}
	\maketitle  


\begin{center}
	Annotation.
\end{center}

{\it \small
We prove, in particular,  the well--known Zaremba conjecture from the theory of continued fractions for any 
prime denominator.
%
%
More precisely, we show, firstly, that under some mild conditions, for any sufficiently large $q$, 
there exists    
$a$  coprime to $q$ such that all  partial quotients of $a/q$ are bounded by 
$O(\sqrt{\log q})$, and,  
moreover we find asymptotically tight lower bound for the number of such $a$. 
Secondly, 
we obtain 
a good 
lower bound for the number $a$  
 such that the sum of all partial quotients of $a/q$ is bounded by $O(\log q \cdot \sqrt{\log \log q})$.
This, accordingly, improves on some results of Korobov and Larcher.
Finally, we show that for all sufficiently large $\mathcal{M}$ there are $\Omega(q^{1-O(1/\mathcal{M})})$ numbers  $a$  coprime to $q$ such that all  partial quotients of $a/q$ are bounded by $\mathcal{M}$. 
}
\\

\section{Introduction}

Let $a$ and $q$ be two positive coprime integers, $0<a<q$. 
By the Euclidean algorithm, a rational $a/q$ can be uniquely represented as a regular continued fraction
\begin{equation}\label{exe}
\frac{a}{q}=[0;c_1,\dots,c_s]=
\cfrac{1}{c_1 +\cfrac{1}{c_2 +\cfrac{1}{c_3+\cdots +\cfrac{1}{c_s}}}}
\,,\qquad c_s \ge 2.
\end{equation}
Assuming $q$ is known, we use $c_j(a)$, $j=1,\ldots,s=s(a)$ to denote the partial quotients of $a/q$; that is,
\begin{equation}\label{def:aq} 
    \frac aq := [ 0; c_1(a),\ldots,c_{s}(a)].
\end{equation}
Also, for any $a$ put 
\begin{equation}\label{def:M(a)} 
    M(a) := \max \{ c_1(a),\ldots,c_{s}(a) \} \,.
\end{equation}

Zaremba's famous conjecture \cite[Page 76]{zaremba1972methode} posits that there is an absolute constant $\zk$ with the following property:
for any positive integer $q$ there exists $a$ coprime to  $q$ such that in the continued fraction expansion (\ref{exe}) all partial quotients are bounded:
\[
c_j (a) \le \zk,\,\, \quad \quad 1\le j  \le s = s(a).
\]
    In other words, there is $a$ coprime to $q$ such that $M(a) \le \zk$. 
Similar questions were considered by Korobov in sixties, see \cite{Korobov_Zaremba} and \cite{Korobov_book}.
In fact, Zaremba conjectured that $\zk=5$.
For large prime $q$, even $ \zk=2$ should be enough, as conjectured by Hensley \cite{hensley_SL2}, \cite{hensley1996}.
This topic is quite  popular, especially recently; see, for example, 
papers 	
\cite{bourgain2011zarembas}--\cite{CP_linear}, 
\cite{chang2011partial}, \cite{hensley1989distribution}--\cite{hensley1996}, 
\cite{KanIV}, 
\cite{Kan_3thms}, 
\cite{Mosh_A+B}, 
\cite{MS_Zaremba_mod}, \cite{Niederreiter_dyadic},  \cite{s_Chevalley}, \cite{Shulga_Zaremba} and many others. 
The history of the question can be found, e.g.,  in \cite{Kontorovich_survey}, \cite{Mosh_survey}, \cite{MMS_popular}. 
We just notice here a remarkable progress of Bourgain  and Kontorovich \cite{bourgain2011zarembas}, \cite{BK_Zaremba} who proved Zaremba's conjecture for ``almost all'' 
denominators 
$q$. 
As for an arbitrary {\it fixed} $q$, using some exponential sums,  Korobov \cite{Korobov_Zaremba} 
proved  in 1960 
that for any  prime  $q$  there is  $a$, 
$1\le a<q$
such that 
\begin{equation}\label{f:Korobov_log}
    M(a) \ll \log q \,.
\end{equation}
The same result takes place for composite $q$, see  \cite{Ruk}.
Our new Theorem \ref{t:Zaremba} below 
gives us an immediate consequence in the classic case $q=p$, started by Korobov,  
\cite{Korobov_Zaremba} and \cite{Korobov_book}$^*$.\footnote{$^*$Korobov pondered prime denominators until his death \cite{PC1}.}

\begin{corollary}
    There is an absolute constant $\mathcal{M}\ge 2$ such that for any prime (or square--free) number $p$ there exists $a$, $(a,p)=1$ such that 
\begin{equation}\label{f:Zaremba_primes}
    \frac{a}{p} = [0;c_1,\dots,c_s] \,, \quad \quad c_j \le \mathcal{M} \,, \quad \quad   j= 1,\dots, s\,.
\end{equation}
\label{cor:Zaremba}
\end{corollary}


Before discussing the proofs of Corollary \ref{cor:Zaremba} and Theorem \ref{t:Zaremba}, 
let us say a few words about the motivation for such questions.
Zaremba's conjecture is connected with some problems of  numerical integration. 
First we state the fundamental result of 1961 due to Koksma and Hlawka (see \cite{Koksma}, \cite{Hlawka} or \cite[Chap. 2]{KN_book}).

	\begin{theorem}
		Let $f: [0,1]^d \to \mathbb{R}$  be a function of bounded variation $\mathrm{V}(f)$ and  $X \subseteq [0,1]^d$ be a finite set. Then 
	\begin{equation}\label{f:VD}
		\left| \int_{[0,1]^d} f(u)\,du - \frac{1}{|X|} \sum_{x\in X} f(x) \right| \le \mathrm{V}(f) \cdot \mathrm{Disc}(X) \,,
	\end{equation}
		where
	\begin{equation}\label{f:Disc}
		\mathrm{Disc}(X)  := \sup_{R = \prod_{i=1}^d [a_i,b_i] } \left| \frac{|X\cap R|}{|X|}- \mathrm{Vol} (R) \right| \,.
	\end{equation}
	\end{theorem}

We call $\mathrm{Disc}(X)$ in \eqref{f:Disc} the {\it discrepancy}   of a sequence of points $X$. It measures the uniform distribution of $X$ in the cube $[0,1]^d$ and, in the light of \eqref{f:VD}, it determines the quality of the approximation of the integral of $f$ by a finite sum over $X$. Thus, the problem reduces to finding sequences $X$ whose discrepancy has as nice an upper bound as possible.


A remarkable theorem due to Schmidt \cite{Schmidt} in the theory of irregularity of distributions gives a lower bound for the discrepancy of the distribution of an {\it arbitrary} finite sequence $X\subseteq [0,1]^2$, namely, $\mathrm{Disc}(X) \gg \frac{\log |X|}{|X|}$.
%
%
%
    Thus, the discrepancy of any sequence cannot be too small.
The Monte Carlo method can be shown to provide only very weak estimates for discrepancy: 
 $\frac{1}{|X|^{1/2+o(1)}} \ll \mathrm{Disc}(X) \ll  \frac{1}{|X|^{1/2-o(1)}}$.
Therefore,  there is a need for efficient construction of sequences $X$ with small $\mathrm{Disc}(X)$.
We content ourselves with the two--dimensional case below.

	In \cite{zaremba1966} Zaremba (and before that Korobov in \cite{Korobov_optimal_I, Korobov_optimal_II})  proposed to consider a discrete winding of a two dimensional torus:
\begin{equation}\label{def:X(a,q)}
	X = X(a,q) = \left\{ \left( \frac{j}{q}, \frac{aj}{q} \right) \right\}_{j=1}^{q} 
	\subseteq [0,1]^2 \,. 
\end{equation}
    Here $a$ and $q$ are, as above, two coprime positive integers such that $0 < a < q$. For $X(a,q)$ the discrepancy can be estimated in terms of the continued fraction expansion of $a/q$ (see \cite[Corollary 5.2]{zaremba1966}).

	\begin{theorem}
		Let $a,q$ be positive coprime integers, $\frac{a}{q} = [0; c_1,\dots,c_s]$
		and let  
        $M=M(a)$. 
		Then 
	\begin{equation}\label{f:Zaremba_D}
		\mathrm{Disc}(X (a,q)) \le \left( \frac{4M}{\log (M+1)} + \frac{4M+1}{\log q} \right) \frac{\log q}{q} \,.
	\end{equation}
    \label{t:Zaremba_D}
	\end{theorem} 
It was later shown that the discrepancy in \eqref{f:Zaremba_D} can be estimated using a more flexible quantity, namely, the {\it sum} of the partial quotients 
\begin{equation}\label{def:S(a)}
    S(a) := 
    \sum_{j=1}^s c_j (a) \,, 
\end{equation}
see 
\cite{Larcher} and  also \cite{KN_book}, \cite{Larcher_survey}. 
Moreover, in terms of the 
sum 
$S(a)$, Larcher  \cite{Larcher}  found a series of further applications for evaluating  the discrepancies of other sequences. For example, if $q$ is a prime number and $g$ is  a primitive root modulo $q$, then the  discrepancies 
\begin{equation}\label{f:Larcher_disc}
    \mathrm{Disc} \left\{ \frac{jg}{q} \right\}_{j=1}^{q} \,,
        \quad \quad 
    \mathrm{Disc} \left\{ \left( \frac{j}{q}, \frac{gj}{q} \right) \right\}_{j=1}^{q} \,,
     \quad \quad 
    \mathrm{Disc} \left\{ \left( \frac{g^j}{q}, \frac{g^{j+1}}{q} \right) \right\}_{j=1}^{q} 
\end{equation}
are all bounded by $O(S(g)/q)$. 
The corresponding conjecture on $S(a)$, namely, 
\begin{equation}\label{conj:Moser}
    \min_{a~:~(a,q)=1}\,  S(a) = O(\log q)
\end{equation}
is attributed  to Moser (see \cite{Larcher_survey}). 
Of course, Zaremba's conjecture implies \eqref{conj:Moser}. 
The first result concerning 
the minimum in \eqref{conj:Moser} 
belongs to Larcher, see  \cite[Corollary 3]{Larcher}. 

\begin{theorem}
    Let $q$ be a  positive integer.
    Then there is a positive integer $a$, $(a,q)=1$ such that for 
\[
    \frac{a}{q} = [0;c_1,\dots,c_s] 
\]
    one has 
\begin{equation}\label{f:Larcher_old}
    \frac{1}{\log q} \sum_{j=1}^s c_j \ll 
    \frac{q}{\_phi (q)} \cdot \log \log q \,.
\end{equation}
\label{t:Larcher_old}
\end{theorem}

Theorem \ref{t:Larcher_old} was improved by 
Rukavishnikova \cite{Ruk}, who reduced the factor $\frac{q}{\_phi (q)}$ in \eqref{f:Larcher_old}, and further by Aistleitner--Borda--Hauke who, in particular, found the currently best implicit  constants in \eqref{f:Korobov_log} and  \eqref{f:Larcher_old}; see \cite[Corollary 2, Corollary 4]{ABH_Ruk+} for {\it all} $q$.



Using some 
mixture  
of methods: the Bourgain--Gamburd machine  (see \cite{BG}, \cite{BG_p^n}, \cite{sh_BG} and the survey \cite{sh_non_survey}), 
the affine sieve \cite{BGS_affine} 
both based on the Helfgott expansion result \cite{Helfgott_growth}
as well as an exact  structural result from \cite{Mosh_A+B} (see Lemma \ref{l:A+B} below) 
we recently improved Korobov's bound \eqref{f:Korobov_log} in \cite{MMS_Korobov}. 
The most important case when $q$ is a (sufficiently large) prime number and the reader, not wishing to think about the peculiarities  of the growth in $\SL_2 (\Z/q\Z)$ for composite $q$, can assume that $q$ is a prime.

\begin{theorem}
	Let $q$ be a  
	positive integer 
	with sufficiently large prime factors. 
	Then there is a positive integer $a$, $(a,q)=1$ and 
	\begin{equation}\label{f:main_M}
		M= O(\log q/\log \log q)
	\end{equation}
	such that 
	\begin{equation}\label{f:main_expansion}
		\frac{a}{q} = [0;c_1,\dots,c_s] \,, \quad \quad c_j \le M\,, \quad \quad   j= 1,\dots, s\,.
	\end{equation}
	Also, if $q$ is a sufficiently large square--free number, then \eqref{f:main_M}, \eqref{f:main_expansion} take place.\\ 
	Finally, if $q=p^n$, $p$ is an arbitrary prime, then \eqref{f:main_M}, \eqref{f:main_expansion}  hold for sufficiently large $n$. 
	\label{t:main}
\end{theorem}

In our first new result, we improve Theorem \ref{t:Larcher_old} and, moreover,  obtain 
a good 
lower bound \eqref{f:Larcher_number} for the number $a$ 
with small 
$S(a)$. 
By $\mathcal{Z}$ denote the collection of all  denominators $q$ satisfying  the conditions of Theorem \ref{t:main}. 
In other words, $q\in \mathcal{Z}$ if one of the following three conditions is satisfied:{\medskip}\\ 
$\bullet$ $q=p^{\a_1}_1 \dots p^{\a_s}_s$, where $p_j$ are sufficiently large primes.\\
$\bullet$ $q=p_1 \dots p_s$, where $p_j$ are primes and $q$ is sufficiently large.\\
$\bullet$ $q= p^n$, where $p$ is any prime and $n$ is a sufficiently large number.
{\medskip}\\ 
In particular,  we assume that $q$ is a sufficiently large number.
Here and everywhere, by ``sufficiently large'' we mean the existence  an absolute constant such that $q$ (or other quantities) is greater than this constant.
Note that the set $\mathcal{Z}$ has positive density (at least $6/\pi^2$, but not density equal to one, as in the ``almost all'' result of Bourgain  and Kontorovich, see \cite{bourgain2011zarembas}, \cite{BK_Zaremba}) and includes all
sufficiently large 
primes. 
Finally, 
the important quantity 
$w_M$ is the the Hausdorff dimension of the set 
$\{ \a = [0; c_1,c_2,\dots] \in [0,1] ~:~ \forall c_j \le M \}$, see  formula  \eqref{f:Hensley_HD} below.

\begin{theorem}
    Let $q\in \mathcal{Z}$ be a positive integer and 
\[
    M = \sqrt{\log \log q} 
    \,.
\]
    Then there are at least 
\begin{equation}\label{f:Larcher_number}
    q^{2w_M-1} 
    \cdot \exp \left( \Omega \left(\frac{\log q}{\log \log q} \right) \right) 
\end{equation}
positive integers $a$, $(a,q)=1$ such that for 
\[
    \frac{a}{q} = [0;c_1,\dots,c_s] 
\]
    one has 
\begin{equation}\label{f:Larcher}
    \frac{1}{\log q} \sum_{j=1}^s c_j \ll 
    \sqrt{\log \log q} \,.
\end{equation}
\label{t:Larcher}
\end{theorem}


In addition to  the new bounds for the discrepancies \eqref{f:Larcher_disc}
Theorem \ref{t:Larcher} gives a series of applications to some questions in  pure combinatorics, namely, 
new  bounds for the number  of spanning trees in simple graphs \cite{CKP_trees_Zaremba},  \cite{CPS_loglog}, 
a phase transition for independent sets of graphs \cite{CSP_Zaremba}, 
and for linear extensions 
\cite{KS_linear}, \cite{CP_linear} (see also recent paper \cite{CKP_planar_CF2}). 
There are other applications, for example to McMullen's conjecture \cite[Conjecture 6.2]{Mcmullen} on real quadratic extensions, see \cite[Theorem 1.11]{Mercat_Z}.

\bp 

As a byproduct of our argument 
we significantly improve Theorem \ref{t:main} and this is our second main result. 

\begin{theorem}
    Let 
    $q\in \mathcal{Z}$ be a  positive integer and 
\begin{equation}\label{f:main_M_new}
        M \ge C \sqrt{\log q}
\,,
\end{equation}
    where $C>1$ is an absolute constant. 
    Then there are at least 
\begin{equation}\label{f:main_new_number}    
    q^{2w_M-1-o(1)} \cdot 
\end{equation}
    positive integers $a$, $(a,q)=1$ with 
    $M(a) \le M+2$. 
\label{t:main_new}
\end{theorem}


%
%
%
The expected number of numerators $a$ such that $M(a)\le M$  is 
$q^{2w_M-1-o(1)}$ 
similar to \eqref{f:Larcher_number} and \eqref{f:main_new_number}, see  \cite{hensley_SL2}, \cite{hensley1996},  \cite{Kontorovich_survey} or just discussion in Remark \ref{r:optimality} below. 
The constant $2$ in $M+2$ is a technical point that can be easily improved, see Lemma \ref{l:M_crit}  and Remark \ref{r:M_crit} below.
So, our methods allow to obtain the same lower bound for the number of such numerators  (an optimal upper bound can be found in  \cite{MMS_popular}).
Finally, note that the 
lower 
estimate 
on $M$ in \eqref{f:main_M_new} is 
close to the 
optimal within the framework of the method,  see Remark \ref{r:optimality} below.
Furthemore, 
we emphasize 
that there is a huge difference between  regimes $M=o(\sqrt{\log q})$ and $M=\Omega(\sqrt{\log q})$ as the remark shows us. 
Namely, if $M=\Omega(\sqrt{\log q})$ the number of $a$ such that $M(a) \ll M$ is comparable with the size of the larger set: 
\begin{equation}\label{f:middle_int}
    \left\{ \frac{a}{q}=[ 0; c_1(a),\ldots,c_{s}(a)] ~:~ c_j (a) \ll M,\,~~~ \forall j\notin s\cdot \left[ \frac{1}{2}-\eps, \frac{1}{2}+\eps \right] \right\} \,,
\end{equation}
where $\eps \sim 1/M$. Therefore, in the case $M=\Omega(\sqrt{\log q})$ we have a new interesting phenomenon: if partial quotients $c_j(a)$ are bounded  by  $O(M)$ for all $j$ except the middle interval $s\cdot [1/2-\eps,1/2+\eps]$, then a positive proportion of such $a$ has {\it all} partial quotients bounded by $O(M)$. 
This sounds unexpected, but on the other hand, it immediately becomes clear that the numbers $a$ from \eqref{f:middle_int} have rather special properties (see, for example, Proposition \ref{p:repulsion} or formula \eqref{f:rough_G}), and its number coincides with 
\eqref{f:main_new_number}. 
On the contrary, if $M=o(\sqrt{\log q})$, then the size of the set \eqref{f:middle_int} significantly exceeds the number of Zaremba numerators $a$, and hence this is a poor approximation to the problem.

In the proof we follow the argument from \cite{MMS_Korobov}, where a huge 
part of the work was done. 
In this paper we add some  new ideas.
First of all, we exploit Cantor's  structure of the set of rationals with bounded partial quotients and show that 
our set of Zaremba's numerators $a$ lives in a certain Ahlfors--David set $\mathcal{AD}$, see Section \ref{sec:AD} and papers \cite{DZ_AD},  \cite{BD_Iab}, \cite{BD_AD_def}. 
While   the irregularity of Ahlfors--David sets in the spirit of \cite{DZ_AD}, \cite{BD_Iab}, \cite{Fraser_almostAP} is useful in Zaremba--type questions, in Section \ref{sec:X} we introduce an even more powerful Diophantine idea that may be interesting in its own right.
Let $J$ be an interval of our  Ahlfors--David set $\mathcal{AD}$,  and we are interested in the 
denominators $x=x(a)$ (let $x\le \sqrt{q}$, say)  of some convergents to $a/q \in J$, which we call  critical denominators. 
Since the elements of $J$ are very close enough to each other (both in the Euclidean sense and in the continued fraction expansion sense), it follows (see Section \ref{sec:X}) that the behavior of the various $x(a_k)$, where $a_k \in J$, resembles the behavior of the denominators $x(a)$ for a {\it single} $a\in J$.
In particular, it can be shown that the set of critical denominators $x(a)$, $a\in J$ is quite  sparse and even ``independent'' in a certain  sense, see Lemmas \ref{l:k=2}, \ref{l:k=3} and Proposition \ref{p:number_x}. 
This independency of $x(a)$,  combined with some simple Diophantine problems of dual  nature,  allow us to 
eliminate 
the denominators $x(a)$ that are too close to $\sqrt{q}$ and obtain our main result.


\begin{theorem}
    There is an absolute constant $\mathcal{M}$
    such that for all $M \ge \mathcal{M}$ 
    and 
    any  $q \in \mathcal{Z}$ 
   there are at least 
\begin{equation}\label{f:main_new_number_Zaremba}
    q^{2w_M-1-o(1)}  
\end{equation}
     positive integers $a$, $(a,q)=1$  
    with 
\begin{equation}\label{f:Zaremba}
    \frac{a}{q} = [0;c_1,\dots,c_s] \,, \quad \quad c_j \le 100 M \,, \quad \quad   j= 1,\dots, s\,.
\end{equation}
\label{t:Zaremba}
\end{theorem}



In the proof of Theorem \ref{t:Zaremba} we also follow the argument from \cite{MMS_Korobov}, as well as the proofs of Theorems \ref{t:Larcher}, \ref{t:main_new}.
The constant $\mathcal{M}$ in Theorem \ref{t:Zaremba} is large,  although  computable; see the approximate calculations in Section \ref{sec:appendix}.
  The main losses arise  from the results concerning the growth in the modular group, where some methods of multiplicative combinatorics are used (namely, as we said above: the Bourgain--Gamburd machine, the affine sieve, and  the Helfgott expansion result). 
Interestingly, our method would give a much better upper bound on $\mathcal{M}$ (though still far from $\mathcal{M}=5$), provided that uniform distribution results can be obtained for the set $A^{-1}$, where $q$ is  a prime number, $A\subseteq \Z/q\Z$ is an arbitrary collection of intervals of length $q^{1/9 - o_{\mathcal{M}}(1)}$, see estimates   
\eqref{cond:eps_upper}, \eqref{cond:eps_upper_final} 
below.
In particular, if we want to find only one  $a$ such that \eqref{f:Zaremba} holds, then the so--called first stage of the Bourgain--Gamburd  machine is not required
(but this is still necessary in the case of a general composite $q$).  In other words, to prove Theorem \ref{t:Zaremba} (but not Theorems \ref{t:main}, \ref{t:Larcher}, \ref{t:main_new}) one does not in principle need the full power of multiplicative combinatorics (it works effectively with intervals of any size), and we can limit ourselves to working with intervals of very specific (but still short) length.
In conclusion, it should be noted that although the constant $1/9$ can be improved, it is still much smaller than the square--root barrier $1/2$.

The expected number of numerators $a$ 
in \eqref{f:Zaremba} is $q^{2w_{M}-1 - o(1)}$ similar to \eqref{f:Larcher_number} and \eqref{f:main_new_number}, see  \cite{hensley_SL2}, \cite{hensley1996},  \cite{Kontorovich_survey} or just Remark \ref{r:optimality}. 
The existence of $q^{1-o(1)}$ such $a$ is important for some applications see, for example, McMullen's Conjecture 6.2 in \cite{Mcmullen}. 
In contrast to Theorems \ref{t:Larcher}, \ref{t:main_new} in Theorem \ref{t:Zaremba} 
our methods allow us to obtain a lower bound of the form 
$q^{1-O(1/M)}$. 
This is, of course, worse than the lower bounds in \eqref{f:Larcher_number} and  \eqref{f:main_new_number} and {\it cannot} be explained by the technical Lemma \ref{l:M_crit}.
Although 
some improvements are possible, 
it would be interesting to obtain the correct order for the number of Zaremba's numerators $a$, see Problem \ref{pr:correct_a}.



The signs $\ll$ and $\gg$ are the usual Vinogradov symbols. 
Let us denote by $[n]$ the set $\{1,2,\dots, n\}$.
For two finite subsets $A,B$ of an abelian group denote by $A+B$ its sumset $\{a+b ~:~ a\in A,b\in B\}$ and if $|A+B|=|A||B|$, then we write $A\dotplus B$. We  use the notation $S(x)$ for the characteristic function of a set $S$. 
All logarithms are to base $e$.

The author would like to thank Max Planck Institute for Mathematics (Bonn) for providing excellent working conditions.
Also, he thanks Nikolay Moshchevitin for his faith in this method and his support, Alisa Sedunova for her advices on Lemma \ref{l:coprimality}, 
and Igor Pak for useful discussions concerning possible combinatorial applications of the results of this paper.

\begin{remark}
    This paper was submitted in late 2025 and posted on arXiv in March 2026. In early May, Xin Zhang informed the author that the Bougain--Gamburd--Helfgott--Sarnak--Varj\'u   theory had been significantly advanced by him in collaboration with Jincheng Tang (see \cite{Tang_Zhang_super_app}, \cite{Tang_Zhang_SP}); consequently, Lemma \ref{l:T-action} holds for any $q$ and any bounded $M$. It follows that Corollary \ref{cor:Zaremba} is valid for any $q$ (although the remaining results of the paper remain unchanged), and thus we now possess a proof of Zaremba's conjecture in full generality, see \cite{Zhang_Zaremba_q}.
\end{remark}

\section{Preliminary results}
\label{sec:preliminary}


In this section we present some notations and preliminary results, following mainly the paper \cite{MMS_Korobov}. 
Formally, the new useful result is Lemma \ref{l:A_A*_str} below, concerning the uniform distribution of the inverse set $A^{-1}$ of a Cantor--type set $A$ in itself. 

We 
start with 
a well--known lemma, see  \cite[Lemma 5, pages 25--27]{Korobov_book} or \cite[Section 9]{Mosh_A+B} or \cite[Lemma 2]{MMS_Korobov}. 
It says that, basically, the partial quotients of a rational number are controlled via the hyperbola $x|y| = q/M$.

\begin{lemma}
    Let $a$ be coprime with $q$ and  $a/q = [0;c_1,\dots,c_s]$. 
    Consider the equation
\begin{equation}\label{eq:M_crit}
    ax \equiv y \pmod q \,, \quad \quad 1\le x<q \,, \quad  1\le |y|<q \,.
\end{equation}
    If for all solutions $(x,y)$ of the equation above one has $x|y| \ge q/M$, then $c_j \le M$, $j\in [s]$. 
    On the other hand, if for all $j\in [s]$ the following holds $c_j \le M$, then all solutions $(x,y)$ of \eqref{eq:M_crit} satisfy $x|y|\ge q/4M$. 
\label{l:M_crit}
\end{lemma}


\begin{remark}
\label{r:M_crit}
    The constant $4$ in lemma above is a technical point that can be easily improved.
    For example, $4M$ can be replaced by  $M+2$; see the proof of \cite [Lemma]{MMS_Korobov}.
\end{remark}

Let $1\le t\le \sqrt{q}$ be a real number. 
Having 
a 
rational number  $\frac{a}{q} = [0;c_1,\dots,c_s]= \frac{p_s}{q_s}$, we write $\frac{p_\nu}{q_\nu}$ for its $\nu$--th convergent.
Let  $M\ge 1$ be an integer.
Define 
\begin{equation}\label{def:Z_M(t)}
    Z_M (t) = \left\{ \frac{a}{q} = [0;c_1,\dots,c_s] ~:~ c_j \le M \mbox{ for all } j \mbox{ s.t. } q_j < t \right\} \,.
\end{equation}
 With some abuse of the notation we denote by the same letter $Z_M(t)$ the set of the {\it numerators} $a\in [q]$, $(a,q)=1$ from \eqref{def:Z_M(t)}. 
 In this sense $Z_M (t)$ can be considered as  a subset of $\Z/q\Z$. 
Further, following \cite{Mosh_A+B} define 
\begin{equation}\label{def:Q_M(t)}
    Q_M (t) = \left\{ \frac{u}{v} = [0;c_1,\dots,c_s] ~:~ c_j \le M,\, \forall j\in [s],\, v < t \right\} \,,
\end{equation}
and
\begin{equation}\label{def:Q'_M(t)}
    \overline{Q_M (t)} = \left\{ \frac{u}{v} = [0;c_1,\dots,c_s] \in Q_M (t)  ~:~ \mathcal{K}( c_1,\dots,c_s, 1 ) \ge t \right\} \,,
\end{equation}
where by $\mathcal{K}(d_1,\dots,d_k)$ we have denoted the correspondent continuant, see \cite{Hinchin}.
In other words, if we consider $\frac{u}{v} \in Q_M (t)$ and $\frac{u'}{v'}:=[0;c_1,\dots,c_s,1]$, then $v'\ge t$.   
Also, let $Q^c_M (t)$ 
be the complement to $Q_M (t)$.
It should be noted 
that 
the structures of the sets $Q^c_M (t)$
 and  $Q_M (t)$
 are completely  different.
Indeed, the set $Q_M(t)$ is of Cantor type (see Section \ref{sec:AD}), on the contrary, it is easy to see that any interval in $\Z/q\Z$ contains a point from $Q^c_M (t)$
 and, consequently, this set can be considered ``everywhere dense''.

The sets $\overline{Q_M (t)}$ and $Z_M (t)$ are closely connected to each other, see \cite{Mosh_A+B}, Lemma \ref{l:A+B} and Section \ref{sec:AD} below. 
%
To formulate further results we need one more definition from the real setting. 
Let  $M\ge 1$ be an integer.
Consider the set of {\it real} numbers  $F_M$, having all partial quotients  bounded by $M$. It is well--known \cite{Hinchin}, 
that for any  $M$ the Lebesgue measure of the set  $F_M$ is zero and its Hausdorff dimension  $w_M:=\mathcal{HD} (F_M)$ is  $w_M = 1-O(1/M)$, as  $M\to \infty$. Good bounds and asymptotic formulae on $w_M$ are contained in papers 
\cite{hensley1989distribution}, \cite{hensley1992continued}. 
In particular, in \cite{hensley1992continued} Hensley 
obtained 
the following result.

\begin{theorem}
		For 
		$M\to \infty$ 
		one has  
\begin{equation}\label{f:Hensley_HD}
		w_M := \mathcal{HD} \left(\{ \a = [0; c_1,c_2,\dots] \in [0,1] ~:~ \forall c_j \le M \} \right) 
		=
		1-\frac{6}{\pi^2 M} - \frac{72 \log M}{\pi^4 M^2} + O\left( \frac{1}{M^2} \right)
		\,. 
\end{equation}
	\label{t:Hensley_HD}
	\end{theorem}

%
%
The 
next result is contained in \cite[Lemma 3]{MMS_Korobov}
(it is a combination of Lemma 2 and Lemma 3 of \cite{Mosh_A+B}). 

\begin{lemma}
    Let $t \le \sqrt{q}$. 
    Then for some absolute constants $c_1,c_2>0$ one has 
\[
    Z_M (t) = I_1 \bigsqcup \dots \bigsqcup I_T \,,
    \quad \quad c_1 t^{2w_M} \le T\le c_2 t^{2w_M} \,,
\]
where $I_j$ are some disjoint  intervals and for all $j\in [T]$ the following holds 
\[
    [q/t^2] \le |I_j| \le \left(\frac{8 (M+1) q}{t^2} +1 \right) 
    \,.
\]
\label{l:A+B}
\end{lemma}

    As we know from Lemma \ref{l:A+B} (also, see definitions \eqref{def:Z_M(t)}---\eqref{def:Q'_M(t)}) the set  $Z_M (t)$ is a Cantor--type set,  more precisely, it is the intersection of the progression (fractions $b/q$ can be reducible) 
    $$
    R_q := \left\{ \frac{b}{q} ~:~ b\in [q]
    \right\}$$ 
    with the  Cantor--type set $Q_M (t)$. 
    In other words, it is a disjoint union of some intervals $J_{u/v}$ of length at least $N = [q/t^2]$ and these intervals are indexed by the set $\overline{Q_M (t)}$.
    All details can be found in \cite[Lemmas 2,3]{Mosh_A+B}.
    It gives us the decomposition 
    \begin{equation}\label{f:Q_M_via_J}
    Q_M (t) = \bigsqcup_{u,v ~:~ \frac{u}{v} \in \overline{Q_M (t)}} J_{u/v} \,, 
    \end{equation}
    where 
    the 
    interval $J_{u/v}$, $\frac{u}{v}=[0;r_1,\dots,r_l] = \frac{p_l}{q_l}$, $r_l\ge 2$ has  ends 
\begin{equation}\label{f:u',v'}
    \frac{u'}{v'}=[0;r_1,\dots,r_{l-1},M+1]\,,  
    \quad \quad 
    \frac{u''}{v''}=[0;r_1,\dots,r_{l-1},r_l-1,M+1] \,.
\end{equation}
    It is easy to see (or consult \cite[formula (7)]{Mosh_A+B}) that $v<v',v''<(M+1)v$ and 
\begin{equation}\label{f:J_uv}
    \frac{1}{v^{2}} <|J_{u/v}|< \frac{2(M+1)}{v^{2}} \,.
\end{equation}
 Also, recall the domino identity for  continuants which we will use above 
\[
    \mathcal{K}_n(x_1,\dots, x_n) = \mathcal{K}_m (x_1,\dots,x_m) \mathcal{K}_{n-m} (x_{m+1}, \dots, x_n) 
    + 
\]
\begin{equation}\label{f:domino}
    \mathcal{K}_{m-1} (x_1,\dots,x_{m-1}) \mathcal{K}_{n-m-1} (x_{m+2}, \dots, x_n) \,.
\end{equation}
Further information about continued fractions can be found in \cite{Hinchin}.

\bigskip



We need the crucial result from \cite[Lemma 4]{MMS_Korobov} (for prime $q$ this  result is actually contained in \cite[Proposition 7]{MMS_popular}). 
The proof 
is 
an application of the Bourgain--Gamburd machine \cite{BG}, based on Helfgott's expansion result \cite{Helfgott_growth}.
Recall that by $\mathcal{Z}$ we denoted the set of all denominators $q$ satisfying the conditions of Theorem \ref{t:main}.

\begin{lemma}
    Let $q \in \mathcal{Z}$, $N$ be a positive integer,  and  
    $A,B \subseteq \Z/q\Z$ be sets.
    Then there is an absolute constant $\kappa>0$ such that 
\begin{equation}\label{f:T-action}
    |\{ (a+2c)(b+2c) = 1 ~:~ a\in A,\, b\in B,\, c \in [N] \}| 
    - \frac{N|A||B|}{q} 
    \ll \sqrt{|A||B|} N^{1-\kappa} \,.
\end{equation}
\label{l:T-action}
\end{lemma}

    Lemma \ref{l:T-action} was generalized in \cite[Theorem 3]{sh_BG} (in \cite{sh_BG} we considered the case of prime $q$ but the argument of \cite{MMS_Korobov} gives us the same for general $q \in \mathcal{Z}$). 
    Below, the constant $\kappa <1$ may vary from line to line.

\begin{lemma}
    Let $q \in \mathcal{Z}$, $\delta \in (0,1]$, 
    $N\ge 1$ be a sufficiently large integer, $N\le q^{c\delta}$ for an absolute constant $c>0$, 
    $A,B\subseteq \Z/q\Z$ be sets, and $g\in \SL_2 (\Z/q\Z)$ be a nonlinear map. 
    Suppose that $S$ is a set, 
    $S \subseteq [N]\times [N]$, $|S| \ge N^{1+\delta}$.
    Then there is 
    a
    constant $\kappa = \kappa (\d) >0$ such that 
\begin{equation}\label{f:BG_new}
    |\{ g(\a+a) = \beta+b ~:~ (\a,\beta) \in S,\,a\in A,\, b\in B \}|
    - \frac{|S||A||B|}{q}
    \ll_g \sqrt{|A||B|} |S|^{1-\kappa} \,.
\end{equation}
\label{l:BG_new}
\end{lemma}

    Taking $S= [N]\times [N]$, $\d=1$ and $gx=1/x$, we 
    get 
    an analogue of Lemma \ref{l:T-action} for the correspondent  two--dimensional family of modular transformations.
    The proof of Lemma \ref{l:BG_new} does not use Kesten's Lemma \cite{Kesten} and free group computations in the so-called first stage of the Bourgain--Gamburd machine (see \cite{BG} and \cite[Lemma 12]{MMS_popular}).

\bigskip 

Lemmas \ref{l:T-action}, \ref{l:BG_new} give us a direct consequence. 

\begin{corollary}   
     Let $q \in \mathcal{Z}$, 
     $\Lambda_1,\Lambda_2 \subset \Z/q\Z$ be sets, $I=[N]$ and  $A=I \dotplus \Lambda_1$, 
     $B=I \dotplus \Lambda_2$.
     Then there is an absolute constant $\kappa >0$ such that 
\begin{equation}\label{f:chi_interval}
    |A\cap B^{-1}| 
    - \frac{|A||B|}{q} 
    \ll 
    \sqrt{|A||B|} N^{-\kappa} \,.
\end{equation} 
\label{c:A_A*_asymptotic}
\end{corollary}
\begin{proof}
    Let $h\sim  \sqrt{N}$ be an integer parameter and write $I(i) = h^{-1} (H*I)(i)+\eps(i)$, where $H=[h]$, $\| \eps\|_\infty =1$ and the function $\eps$ is supported on two  shifts  of the interval $H$.
    We have 
\[
    |A\cap B^{-1}| = \sum_{x\in A\cap B^{-1}} 1 = 
    \sum_{(i+\la)(i'+\la') \equiv 1\pmod q}  \Lambda_1 (\la) \Lambda_2 (\la') I(i) I(i')
\]
\[
    =
    \sum_{(i+\la)(i'+\la') \equiv 1\pmod q}  \Lambda_1 (\la) \Lambda_2 (\la') (h^{-1} (H*I)(i)+\eps(i)) (h^{-1} (H*I)(i')+\eps(i'))
\]
\[
    =
    h^{-2} \sum_{a,a',h_*,h'_* ~:~ (a+h_*)(a'+h'_*)\equiv 1\pmod q} A(a) B(a') H(h_*) H(h'_*) + \mathcal{E}
\]
\begin{equation}\label{tmp:29.04_2-}
    = 
    \frac{|A||B|}{q} 
    + \mathcal{E} + O(\sqrt{|A||B|} N^{-\kappa/2}) \,,
\end{equation}
    where we have used Lemma \ref{l:BG_new}
    and 
    the error term $\mathcal{E}$ can be estimated as  
\[
    |\mathcal{E}| \ll  
    h^{-1} \sum_{a,a',h_*,h'_* ~:~ (a+h_*)(a'+h'_*)\equiv 1\pmod q} (\Lambda_1 (a) B(a') + A(a) \Lambda_2 (a')) 
    H(h_*) H(h'_*)
\]
\[
    +
    \sum_{a,a',h,h' ~:~ (a+h_*)(a'+h'_*)\equiv 1\pmod q} \Lambda_1 (a) \Lambda_2 (a')   H(h_*) H(h'_*)
\]
\begin{equation}\label{tmp:29.04_2}
    \ll 
     \frac{|A||B|}{q} \left( \frac{h}{N} + \frac{h^2}{N^2} \right) + \sqrt{|A||B|} \cdot  \left( \frac{h^{1-\kappa}}{\sqrt{N}} + \frac{h^{2-\kappa}}{N} \right) 
    \ll 
    \frac{h|A||B|}{Nq} + \frac{\sqrt{|A||B|} h^{1-\kappa}}{\sqrt{N}} \,.
\end{equation} 
    Here we have used that $h\le \sqrt{N}$ and applied Lemma \ref{l:BG_new} again. 
    Note that there are better bounds since $A$ is $I$--invariant and not just $H$--invariant, but we do not need such improvements.
    Finally, returning to \eqref{tmp:29.04_2-} and assuming that $\kappa \le 1$, we get 
\[
    |A\cap B^{-1}| 
    - \frac{|A||B|}{q} 
    \ll 
    \sqrt{|A||B|} N^{-\kappa/2} +
    \frac{h|A||B|}{Nq} 
    + \frac{\sqrt{|A||B|} h^{1-\kappa}}{\sqrt{N}} 
    \ll 
    \sqrt{|A||B|} N^{-\kappa/2}
\]
    as required. 
$\hfill\Box$
\end{proof}

\bigskip

Let $q \in \mathcal{Z}$, $\Lambda_1,\Lambda_2 \subset \Z/q\Z$ be sets, $I=[N]$ and  $A=I \dotplus \Lambda_1$, $B=I \dotplus \Lambda_2$.
Corollary \ref{c:A_A*_asymptotic} allows us to compute the size of the intersection $A\cap B^{-1}$, but what is the structure of this set? 
Using 
a suitable variant 
of the Bourgain--Gamburd machine,  as was done in Lemmas \ref{l:T-action}, \ref{l:BG_new}, one can show that the set  $A\cap B^{-1}$ is uniformly distributed inside $A$ in the sense  that  Fourier coefficients  of the characteristic function $(A\cap B^{-1}) (x)$ and the function $A(x) \cdot \frac{|A\cap B^{-1}|}{|A|}$ are close.
However, another application of Lemma \ref{l:T-action}, \ref{l:BG_new} shows that $A\cap B^{-1}$ is uniformly distributed inside $A$ in an even stronger sense.

\begin{lemma}
   Let $q \in \mathcal{Z}$, $\Lambda_1,\Lambda_2 \subset \Z/q\Z$ be sets, $I=[N_1]$ and  $A=I \dotplus \Lambda_1$, $B=I \dotplus \Lambda_2$. 
   Put $C=A\cap B^{-1}$ and 
   \[ 
        \ov{C} = \ov{C} (N_2) = C+[-N_2/2,N_2/2] \,, 
   \]
   where $N_2\le N_1$. 
   Then 
\begin{equation}\label{f:A_A*_str}
    |A|
    - O\left( q 
    \sqrt{\frac{|A|}{|B|}} \cdot N^{-\kappa}_2 \right) 
    \le |\ov{C}| \le |A|\left( 1+ \frac{2N_2}{N_1} \right) 
\end{equation}
\label{l:A_A*_str}
\end{lemma}
\begin{proof}
    The upper bound in \eqref{f:A_A*_str} is trivial and hence it is enough to obtain the lower one. 
    On the one hand,  since $C\subseteq \ov{C}$, then by Corollary \ref{c:A_A*_asymptotic} one has 
\begin{equation}\label{tmp:14.05_1}
   |\ov{C} \cap B^{-1}| \ge |C\cap B^{-1}| = |A\cap B^{-1}|
   = 
   \frac{|A||B|}{q} 
    +
    O(\sqrt{|A||B|} N^{-\kappa}_1) \,.
\end{equation}
    On the other hand, we can use  the same corollary to obtain an upper bound for $|\ov{C} \cap B^{-1}|$.
    Indeed, one can easily assume that $\ov{C}$ is the disjoint union of intervals and therefore Corollary \ref{c:A_A*_asymptotic} is applicable. 
    More precisely, let $h\sim N^{1/2}_2$ be an integer parameter. 
    We have $\ov{C} = \bigsqcup_{j=1}^r I_j$, where intervals $I_j$ have sizes at least $N_2$. 
    Splitting each interval into subintervals of length $h$, we obtain 
\[
    \ov{C} \subseteq C_1 \bigsqcup C_2 \,,
\]
    where $C_1=[h] \dotplus \Lambda$, $C_2=[h] \dotplus \Lambda'$ and $|C_2|\le r h$. 
    In particular, $|C_1| \ge |\ov{C}|/2 \ge r N_2/2$ and therefore
    $|C_2| \le 2h|C_1|/N_2$. 
    By Corollary \ref{c:A_A*_asymptotic} it follows that 
\[
    |\ov{C} \cap B^{-1}| \le \frac{|C_1||B|}{q} + 
    \frac{|C_2||B|}{q} + O(\sqrt{|C_1||B|} h^{-\kappa}) 
\]
\begin{equation}\label{tmp:14.05_2}
    \le 
    \frac{|\ov{C}||B|}{q}
    +
    \frac{2h|C_1||B|}{N_2 q}
    +
    O(\sqrt{|C_1||B|} h^{-\kappa}) 
    =
     \frac{|\ov{C}||B|}{q}
     + 
     O\left(\sqrt{|\ov{C}||B|} N_2^{-\kappa/4} \right) \,.
\end{equation}
    Notice that $|\ov{C}| \ll |A|$ thanks to the condition $N_2\le N_1$. 
    Comparing bounds \eqref{tmp:14.05_1}, \eqref{tmp:14.05_2}, we derive
\[
    |A|- O\left(
    q
    \sqrt{\frac{|A|}{|B|}} \cdot N^{-\kappa/4}_2 \right) 
    \le |\ov{C}| 
\]
    as required. 
%
$\hfill\Box$
\end{proof}

\bigskip 

In particular, Lemma \ref{l:A_A*_str} shows that if $A=B = I \dotplus \Lambda$, $N_1=N_2$ and 
\begin{equation}\label{cond:initial}
    |A| N^{\kappa/2}_2 \gg q^{1+o(1)}  \,,
\end{equation}
then $A^{-1}$ intersects 
\begin{equation}\label{cons:initial}
    |\Lambda| \left( 1- O(N_2^{-\kappa/2}) \right)
\end{equation} 
intervals of the set $A$. 
Indeed, it suffices to split the set $\Lambda = \{\la_1,\dots,\la_s\}$ into two sets with $\la_j$ having even/odd indices and apply the lemma above to each half.
Finally, note that the condition \eqref{cond:initial} is natural, since in this case Corollary \ref{c:A_A*_asymptotic} gives us an asymptotic formula for the size of $A\cap A^{-1}$.
We will use this and even stronger conditions of this kind in Section \ref{sec:proof}. 

In the same spirit, let us formulate another result regarding $A\cap A^{-1}$.
Let $A = I \dotplus \Lambda$, $|I|=N$ and $N_*\le N$ be an integer parameter. 
Split any interval $I+\la$, $\la \in \Lambda$ into subintervals of length $N_*$. 
We say that an interval $I+\la$, $\la \in \Lambda$ is {\it good} (or {\it $N_*$--good}) if $A^{-1}$ intersects at least $\frac{|I|}{N_*} \cdot (1-O(N^{-\kappa/2}_*))$ of subintervals of length $N_*$, and {\it bad} otherwise. 
It is easy to see that any $N_*$--good interval is automatically $2N_*$--good. 

\begin{lemma}
   Let $q \in \mathcal{Z}$, $A = I \dotplus \Lambda$, $|I|=N$, $N_* \le N$ and 
\[
    \tilde{A} = \bigsqcup_{\la \in \Lambda,\, I+\la \mbox{ is good }} (I+\la) \subseteq A \,. 
\]
    Then 
\begin{equation}\label{f:dense_subint}
    |A\setminus \tilde{A}| \ll  q N^{-\kappa/2}_* \,.
\end{equation}
\label{l:dense_subint}
\end{lemma}
\begin{proof}
    The set $A$ is splitted  into some subintervals of length $N_*$.
    By Lemma \ref{l:A_A*_str} (or see inequality \eqref{cons:initial}) the collection of such $N_*$--bad intervals forms a set of size $O(q N^{-\kappa}_*)$.
    Hence
\[
    \sum_{\la \in \Lambda,\, I+\la  \mbox{ is }N_*-\mbox{bad }} |I+\la| \cdot  N^{-\kappa/2}_* \ll q N^{-\kappa}_*
\]
    as required. 
    This concludes the proof.  
$\hfill\Box$
\end{proof}

\bp 

In our main application $A=Z_M (t)$ and henceforth we will work with the set 
$\tilde{A}$ instead of $A$.
With some violation of notation, we will denote this set by the same letter $A$.
We have conditions of the form \eqref{cond:initial} or, for example, \eqref{f:eq_J_cond_q}, see below, and hence removing such a tiny set $A\setminus \tilde{A}$, $|A\setminus \tilde{A}|\le |A|N^{-\kappa/4}_*$, say,  does not destroy the {\it interval structure} of $A$ (for example, removing a positive proportion of elements of $A$ can destroy this structure). 
Indeed, a half of the set $\tilde{A}$ still can be written as $\tilde{I} \dotplus \tilde{\Lambda}$, where $\tilde{I}$ is the interval of length $N^{\kappa/2}_*$, say.

Finally, let us remark that given a good interval $I+\la = S \dotplus I_*$, where $|I_*|=N_*$, one can see that the massive set of shifts $S$ contains a lot of combinatorial configurations.
For example, $S$ contains an arithmetic progression of length $\Omega(N^{\kappa/2})$
or solutions of any affine linear equation.
We will use such  type of observations in Section \ref{sec:X}  and more precisely we need  the following simple lemma.
Giving a set $A\subseteq [N]$, $|A| = \d N$ and   a positive integer  $k\ge 2$ we say that an interval $J\subseteq [N]$ is 
{\it $k$--equidistributed} if, decomposing $J$ into $k$ subintervals $J=J_1 \bigsqcup \dots \bigsqcup J_k$, $|J_i| \ge \lfloor |J|/k \rfloor$, the condition  $|A\cap J_i|\ge \d |J_i|/2$ holds  for all $i\in [k]$.

\begin{lemma}
    Let $A\subseteq [N]$ be a set, $|A| = \d N$ and $k\ge 2$ be  a positive integer.
    Then there is $k$--equidistributed interval $J\subseteq [N]$ such that $|J|\ge N \cdot \exp (-4k \log (4k) \log (1/\d))$.
\label{l:equidistributed}
\end{lemma}
\begin{proof}
    If $[N]$ is $k$--equidistributed, then we are done, otherwise $[N] = J_1 \bigsqcup \dots \bigsqcup J_k$,  $|J_i| \in \{ \lfloor N/k \rfloor, \lceil N/k \rceil\}$  and for some $i\in [k]$ one has $|J_i \cap A| \le \d |J_i|/2$.
    Hence
\[
    \frac{|A\cap ([N] \setminus J_i)|}{N-|J_i|} \ge 
    \frac{\d N - 2^{-1}\d \lceil N/k \rceil}{N-\lceil N/k \rceil} 
    \ge  \d \left( 1+ \frac{1}{4k} \right) \,,
\]
    provided $N\ge 2k$. 
    Of course we can assume the last condition because otherwise there is nothing to prove. 
    It follows that for some $l\in [k]$, $l\neq i$ the  density of $A$ in $J_l$ is at least $\d \left( 1+ \frac{1}{4k} \right)$. 
    Apply the same argument to the interval $J_l$. And so on. 
    Clearly, our algorithm must stop after at most $s:=4k \log (1/\d)$ number of  steps. 
    Also, $|J| \ge N/(4k)^s \ge  N \cdot \exp (-4k \log (4k) \log (1/\d))$ as required. 
$\hfill\Box$
\end{proof}


\bp 

We finish this section by a simple technical  Lemma \ref{l:coprimality} which we will use later. 
Put $t:=q^{1/2-\eps}$,
and let $N:=q^{2\eps}$, so 
$t=\sqrt{q/N}$. 
Also, write $A=Z_M(t)$ for brevity and for any $g|q$, $g\le q/t$ define
$A_g = \{ a\in A ~:~ g|a \}$. 
In  our method we count the number of $a$ such that 
\begin{equation}\label{f:reduced_fractions}
    \sigma (A,A) = \left| \left\{ \frac{a}{q} ~:~ \frac{a}{q} = [0;c_1,\dots, c_s],\, c_j \le M \right\} \right|
\end{equation}
and we 
prove that $\sigma(A,A) = \frac{|A|^2}{q} (1+O(N^{-\kappa}))$, provided $\eps \gg M^{-1}$, see Section \ref{sec:proof} below, 
but the problem is that the fractions $\frac{a}{q}$  in \eqref{f:reduced_fractions} can be reduced. 
Recall that for a composite $q$ the set $A^{-1}$ is the collection 
$$\{ a\in [q] ~:~(a/(a,q))^{-1}/(q/(a,q)) \in A\} \,,$$ 
where $(a/(a,q))^{-1}$ is the inverse element modulo $q/(a,q) \ge t$, see details in \cite{Mosh_A+B}.
In other words, $\sigma (A,A)$ is the size of the intersection of $A$ and $A^{-1}$ defined this way.  
So, indeed we must exclude reduced fractions $\frac{a}{q}$  from \eqref{f:reduced_fractions} to guarantee that $a$ in Theorems 
\ref{t:Larcher}, 
\ref{t:main_new} and  \ref{t:Zaremba} can be taken coprime to $q$. 
Not surprising that such results must follow from a suitable application of the M\"obius inverse formula.  
Let $\sigma_* (A,A)$ be the desired  collection of $a$ in \eqref{f:reduced_fractions} such that $(a,q)=1$. 
We have $\sigma_* (A,A) = \sum_{g|q} \mu (g) \sigma(A_g,A)$.


\begin{lemma}
    Let $M\ge 2$, $c\in (0,1)$ be a real number, 
    $\eps \gg M^{-1}$
    and $M \ll \log^{1-c} q$. 
    Then  
\begin{equation}\label{f:coprimality}
    \sigma_* (A,A) \ge \frac{|A|^2 \_phi (q)}{q^2} \cdot (1-o_c (1)) \gg \frac{|A|^2}{q^{1+o(1)}}\,.
\end{equation}
    The same it true for any set $A = I \dotplus \Lambda$, $|I|=N$, provided $\exp(O(\omega^c (q))) \ll N^{\kappa c/16}$. 
\label{l:coprimality}
\end{lemma}
\begin{proof}
    We prove only the first part of the result, since the second follows similarly.
    Let $\kappa>0$ be the absolute constant from Corollary \ref{c:A_A*_asymptotic}. 
    Split the sum $\sigma_* (A,A) = \sigma_1+\sigma_2$, where the sums $\sigma_1, \sigma_2$ are taken over $g\ge N^{\kappa/8} := T$ and $g< T$, correspondingly.
    Using 
    our condition 
    $M \ll \log^{1-c} q$, we obtain 
\begin{equation}\label{tmp:3.09_I--}
    \sigma_1 \ll\sum_{g|q,\, g\ge T} \mu^2 (g) 
    \cdot 
    \frac{q}{g}
    \ll 
    q^{} T^{-c} \sum_{g|q,\, g\ge T} \mu^2 (g) g^{c-1}
\end{equation}
\begin{equation}\label{tmp:3.09_I-}
        \ll 
     q^{} N^{-\kappa c/8}   \exp(O(\omega^c (q)))
    \ll \frac{|A|^2}{q} \cdot N^{-\kappa c/16} \,.
\end{equation}
    In other words, the sum $\sigma_1$ is negligible.   
    It remains to compute $\sigma_2$.
    By formula \eqref{f:Q_M_via_J} we have 
    $|A| = \sum_{\frac{u}{v} \in \overline{Q_M (t)}} |J_{u/v}|$, each $|J_{u/v}| \ge N$, $|A| \gg N t^{2w_M}$, 
    and hence in view of $g< N^{\kappa/8} \le \sqrt{N}$ one has
\begin{equation}\label{tmp:3.09_I}
    |A_g| = g^{-1} \sum_{\frac{u}{v} \in \overline{Q_M (t)}} |J_{u/v}| + \theta |\overline{Q_M (t)}| 
    =
    \frac{|A|}{g} + O(t^{2w_M}) 
    = \frac{|A|}{g} \left( 1 + O(N^{-1/2}) \right) \,.
\end{equation}
    Similarly, for $g< T \le \sqrt{N}$  the set $A_g$ is a collection of arithmetic progressions with step $g$ and length at least $N/T \ge \sqrt{N}$.
    Thus by Corollary \ref{c:A_A*_asymptotic} and formula \eqref{tmp:3.09_I}, we get 
\[
    \sigma (A_g,A) = \frac{|A_g||A|}{q} + O(N^{-\kappa/2} \sqrt{|A_g||A|})
    = 
    \frac{|A_g||A|}{q} \left( 1 + O(N^{-\kappa/4}) \right) \,.
\]
    Therefore, applying the previous formulae and using the same argument as in \eqref{tmp:3.09_I--}, \eqref{tmp:3.09_I-}, we obtain 
\[
    \sigma_2 :=  \sum_{g|q,\, g< T} \mu (g) \sigma(A_g,A)
    = 
    \sum_{g|q,\, g< T} \mu (g) 
    \frac{|A_g||A|}{q} \left( 1 + O(N^{-\kappa/4}) \right)
\]
\[
    = 
    \sum_{g|q,\, g< T} \mu(g) \frac{|A_g||A|}{q} + 
    O\left( \frac{|A|^2}{N^{\kappa/4} q} \sum_{g|q,\, g< T} \mu^2 (g) g^{-1} \right) 
    =
    \frac{|A|^2}{q} \sum_{g|q,\, g<T} \frac{\mu (g)}{g}
    + 
    O\left( \frac{|A|^2}{q} \cdot N^{-\kappa/8} \right)
\]
\[
    =
    \frac{|A|^2 \_phi (q)}{q^2} 
    + 
    \frac{|A|^2}{q} \sum_{g|q,\, g\ge T} \frac{\mu^2 (g)}{g}
    + O\left( \frac{|A|^2}{q} \cdot N^{-\kappa/8} \right)  
\]
\[
    = 
    \frac{|A|^2 \_phi (q)}{q^2} 
    +
    \frac{|A|^2}{T^c q} \sum_{g|q,\, g\ge T} \frac{\mu^2 (g)}{g^{1-c}} 
    + O\left( \frac{|A|^2}{q} \cdot N^{-\kappa/8} \right) 
    = 
    \frac{|A|^2 \_phi (q)}{q^2} 
    +
    O\left( \frac{|A|^2}{q} \cdot N^{-\kappa c/16} \right)  
    \,.
\]
    So, the error term in the last formula is also negligible and we obtain the required result.

    Let us give another even simpler argument that  works
    in the range 
    $M\ll \log \log q$,  
    and it is enough for Theorems \ref{t:Larcher} and \ref{t:Zaremba}. 
    Indeed, thanks to Lemma \ref{l:A+B}, we get (in the previous notation $d:=q/g$) 
\begin{equation}\label{tmp:21.12_1}
    \sigma_1 \le \sum_{d|q,\, d<q/T} \mu^2 (q/d) |A_d| \ll M  \sum_{d|q,\, d<q/T} \mu^2 (q/d) d^{w_M} N^{1-w_M} 
    \ll 
    M T^{-w_M} 2^{\omega (q)} q^{w_M} N^{1-w_M}
\end{equation}
\begin{equation}\label{tmp:21.12_1'}
    \ll 
        M T^{-w_M} 2^{\omega (q)} |A| 
    \ll 
        T^{-w_M/2} |A| 
\end{equation}
    and hence it is negligible. 
    For $d\ge q/T$ use the upper bound for $\sigma_2$. 
    This concludes the proof.  
$\hfill\Box$
\end{proof}

\bp 

Let us emphasize  one more time that the computations in \eqref{tmp:3.09_I--}, \eqref{tmp:3.09_I-} and in \eqref{tmp:21.12_1}, \eqref{tmp:21.12_1'} show that the set of elements $a\in A$ such that 
\begin{equation}\label{cond:gcd}
    (a,q) > N^{c_* \kappa} 
\end{equation}
(here $c_*>0$ is an arbitrary absolute constant) 
is negligible, and the complement enjoys the same interval structure.

\section{On Ahlfors--David sets in the theory of continued fractions}
\label{sec:AD}

In this section we recall some results concerning the structure of the set $Z_M (t)$, namely, we show that $Z_M (t)$ is an Ahlfors--David set, see 
\cite{BD_Iab}, \cite{BD_AD_def}, \cite{Dyatlov_AD}, \cite{DZ_AD}. 
The material in this part of our paper is probably contained  somewhere, but we could not find an exact reference, so we have included it for completeness and to help the reader to understand better the fractal structure of the set $Z_M (t)$. 
Recall that a set $A \subseteq \Z/q\Z$ is an {\it Ahlfors--David set ($N$---Ahlfors--David set) of dimension $w\in [0,1)$} if $A=I \dotplus \Lambda$, where $I=[N]$ is an interval, $\Lambda$  is an arbitrary set and there are absolute constants $C_1,C_2 \ge 1$ such that for any interval $\mathcal{I} \subseteq \Z/q\Z$  
the following holds 
\begin{equation}\label{def:Frostman}
   |A \cap \mathcal{I}| \le C_1 |\mathcal{I}|^{w} N^{1-w} \,,
\end{equation}
    and for any interval $\mathcal{I} \subseteq \Z/q\Z$, $|\mathcal{I}| \ge N$ with the center at the origin and $a\in A$
    one has 
\begin{equation}\label{def:AD}
    |A\cap (\mathcal{I}+a)| \ge C_2^{-1} |\mathcal{I}|^{w} N^{1-w} \,.
\end{equation}

\bigskip

Put $t:=q^{1/2-\eps}$ and let $N:=q^{2\eps}$, so 
$t=\sqrt{q/N}$. 
In this section we consider $Z_M (t)$ as a subset of $\Z/q\Z$.
We are now ready to show that $Z_M (t)$ is an Ahlfors--David set in the sense given above. 

\begin{lemma}
    Let  $M\ge 2$ be an integer, $t\gg M$ be a real number,  and 
    let $\mathcal{I} \subseteq \Z/q\Z$ be an arbitrary  interval. 
    Then 
\begin{equation}\label{f:AD_upper}
    |Z_M (t) \cap \mathcal{I}| \ll M^4 |\mathcal{I}|^{w_M} 
    N^{1-w_M}
    \,.
\end{equation}
    Similarly, if $|\mathcal{I}|\ge N$
    and the center of $\mathcal{I}$ belongs to $Z_M (t)$, then
\begin{equation}\label{f:AD_lower_no_Rq}
    |Z_M (t) \cap \mathcal{I}| \gg  M^{-1} |\mathcal{I}|^{w_M} N^{1-w_M} \,.
\end{equation}
\label{l:AD}
\end{lemma}
\begin{proof}
    Let $\D = |\mathcal{I}| \le q$. 
    Without loosing of the generality we can assume in the first case that $\D \ge N$ (and even $\D \gg M^{4/(1-w_M)} N \gg M^4 N$, provided $\D\le q$) and that the center $a$ of $\mathcal{I}$ belongs to $Z_M(t)$. It allows us to obtain formulae \eqref{f:AD_upper}
    and 
    \eqref{f:AD_lower_no_Rq}
    simultaneously.

    Recall that if we consider the set $Z_M(t)$ as the set of rationals, then 
    $Z_M(t) = Q_M(t) \cap R_q$. 
    Choose $t^2_*=16(M+1)q/\D < t^2$ and find the unique fraction $u_*/v_* \in \overline{Q_M (t_*)}$ such that $u_*/v_*$ is a convergent or a semiconvergent to $a/q$, see \cite[Lemma 1]{Mosh_A+B}. 
    Let us show that 
\begin{equation}\label{f:inclusion_ZM}
    J_{u_*/v_*} \cap Q_M (t) \cap R_q \subseteq Z_M (t) \cap \mathcal{I} \,.
\end{equation}
    Indeed, by the definition of the set  $\overline{Q_M (t_*)}$ one has $t_*/2 \le v_* < t_*$ and since $u_*/v_*$ is a convergent or a semiconvergent to $a/q$, it follows for any $u/v\in J_{u_*/v_*}$ that (see formula \eqref{f:J_uv}) 
\[
    \left| \frac{a}{q} - \frac{u}{v}\right| \le 
    |J_{u_*/v_*}| \le \frac{2(M+1)}{v^2_*} \le \frac{8(M+1)}{t^2_*} \le \frac{\D}{2q} 
\]
    and \eqref{f:inclusion_ZM} follows. 
    Conversely, put $(\tilde{t})^2=q/((M+2)\D) < t_*^2$ 
    and 
    find the unique fraction  $\a = \tilde{u}/\tilde{v} = p_l (\a)/q_l (\a) \in \overline{Q_M (\tilde{t})}$ such that $\tilde{u}/\tilde{v}$ is a convergent or a semiconvergent to $a/q$.
    Let us show that 
\begin{equation}\label{f:exclusion_ZM}
    Z_M (t) \cap \mathcal{I}' \subseteq J_{\tilde{u}/\tilde{v}} \cap Q_M (t) \cap R_q \,,
\end{equation}
    where $\mathcal{I}'$ 
    either left or right 
    half 
    of the interval $\mathcal{I}$. 
    Suppose that there exists   $b/q \in Z_M (t) \cap \mathcal{I}'$ such that  $b/q \notin J_{\tilde{u}/\tilde{v}}$.
    It means that $b/q$ is not a convergent or a semiconvergent to $a/q$.
    Hence
\begin{equation}\label{tmp:upper_I}
    \frac{1}{2(M+2) (\tilde{t})^2} < \frac{1}{2q_l (\a)(q_l (\a) +q_{l+1} (\a))} \le \left| \frac{a}{q} - \frac{b}{q} \right| \le \frac{\D}{2q} 
\end{equation}
    and this is a contradiction. 
    So, we have covered a half of the interval $\mathcal{I}$. 
    Take the nearest point $a_2/q$ in the uncovered part of $Z_M (t) \cap \mathcal{I}$ (if it exists) and apply the previous argument with $a_2/q$. 
    We find another interval $J_{\tilde{u}_2/\tilde{v}_2}$ such that  $\tilde{u}_2/\tilde{v}_2$ is a convergent or a semiconvergent to $a_2/q$. 
    It remains to recall that  
    $J_{\tilde{u}/\tilde{v}} \cap J_{\tilde{u}_2/\tilde{v}_2} = \emptyset$ and then computations in \eqref{tmp:upper_I} show that the union of these two intervals cover the whole set $Z_M (t) \cap \mathcal{I}$.

    Now our task is 
    to count  the number of elements of $R_q$ 
    in the sets \eqref{f:inclusion_ZM}, \eqref{f:exclusion_ZM}.
    For example, consider $J_{u_*/v_*}$, $\frac{u_*}{v_*}=[0;r_1,\dots,r_l] = \frac{p_l}{q_l}$
    and then we need to find the number of fractions $u/v \in Q_M (t)$ which starts with $r_1,\dots, r_{l-1}$ in the continued fraction expansion, see formula \eqref{f:u',v'}. 
    Let 
    $$  
        \frac{u}{v} = [0;r_1,\dots,r_{l-1}, c_l, \dots,c_s] \in Q_M (t) \,.
    $$
    We have $t_* > \mathcal{K}(r_1,\dots,r_{l-1})\ge \frac{t_*}{2(M+1)}$. 
    Using \eqref{f:domino} it follows that 
    \begin{equation}\label{tmp:23.12_1}
    \frac{\mathcal{K}(r_1,\dots, r_{l-1}, c_l,\dots,c_s)}{2\mathcal{K}(r_1,\dots,r_{l-1})}
    \le 
    \mathcal{K}(c_l,\dots,c_s) \le \frac{2 (M+1)t}{t_*} 
    \end{equation}
    and hence, in particular, the ratio $\mathcal{K}(c_{l+1},\dots,c_s)/\mathcal{K}(c_l,\dots,c_s)$ belongs to the set $Q_M (2 (M+1)t/t_*)$. 
    On the other hand, by Lemma \ref{l:A+B} 
    we know that 
    the 
    intervals of $Z_M(t)$ (we know that all of them have lengths at least $N$) indexed by  the 
    continuants $\mathcal{K}(r_1,\dots, r_{l-1}, c_l,\dots,c_s)$.
    Similarly, take continuants 
    $\mathcal{K}(c_l,\dots,c_s)$
    such that the ratio 
     $\mathcal{K}(c_{l+1},\dots,c_s)/\mathcal{K}(c_l,\dots,c_s)$
    belongs to the set 
    $\overline{Q_M (\epsilon t/t_*)} \subseteq Q_M (\epsilon t/t_*)$, where $\epsilon>0$ is a constant. 
    By the first inequality in \eqref{tmp:23.12_1}
    the later set belongs to \eqref{f:inclusion_ZM} for a sufficiently small constant $\epsilon$.
    Thus by Lemma \ref{l:A+B} and by the definition of the number $t_*$ we have at least 
\[
    \left( \frac{\epsilon t}{t_*} \right)^{2w_M} N \gg M^{-1} \D^{w_M} 
    N^{1-w_M}
\]
    fractions 
    in \eqref{f:inclusion_ZM}. 
    Thus we have 
    proved 
    \eqref{f:AD_lower_no_Rq}. 
    The upper bound \eqref{f:AD_upper} can be obtained in a similar way, in which case we work with the set 
    $Q_M (2 (M+1)t/\tilde{t})$.
This completes the proof. 
$\hfill\Box$
\end{proof}



\section{On critical denominators of elements from  $Z_M (t) \cap Z^{-1}_M (t)$}
\label{sec:X}

Put  $t:=q^{1/2-\eps}$, $\eps \in (0,1/2)$, and let $N:=q^{2\eps}$, so 
$t=\sqrt{q/N}$. 
Sometimes we consider the set  $Z_M (t)$ as a subset of 
$\Z$. 
Notice right now that conditions \eqref{cond:k=2}, \eqref{cond:k=3}, \eqref{cond:number_x} below give us an {\it upper} bound for the parameter $\eps$. 
In other words, even at this stage of the proof one has to deal with rather 
short 
intervals where classical analytical methods do not work.

In this section we study the denominators 
\begin{equation}\label{f:x_s(a)_def} 
    x=x(a) = x_s (a) = \mathcal{K} (r_1,r_2,\dots, r_s) \in [t,q/t]
\end{equation}
of some convergents to $a/q$, where 
$a\in Z_M (t) \cap Z^{-1}_M (t)$, 
which we call {\it critical denominators}. 
Given an $a\in Z_M (t) \cap Z^{-1}_M (t)$ one can has several of such $x(a)$ and the collection of them is denoted by $\mathcal{X}(a)$. 
Similarly,  given a set $S\subseteq \Z/q\Z$ we put $\mathcal{X}(S) = \bigcup_{s\in S} \mathcal{X} (s)$. 
We mostly study critical denominators $x$ with the additional 
constraint 
$x\le \sqrt{q}$, which we call {\it critical denominators of Type I}, and if $x>\sqrt{q}$, then we say that $x$ has {\it Type II}.
Below we will see (consult Section \ref{sec:proof}) that critical denominators of Type II are easily transformed into Type I due to the symmetry of our problem.
In a similar way a critical denominator is called {\it $\tilde{M}$--critical denominator} if in addition $x|ax| \le q/\tilde{M}$. 
Notice that 
if there are no critical denominators, then by Lemma \ref{l:M_crit} all partial quotients of $a$ are bounded by $\max\{ M,\tilde{M}\}$.
So, if the reader does not interested in obtaining correct number of $a$ in Theorems \ref{t:Larcher}, \ref{t:main_new}, then it is possible to assume 
\[
    \mbox{for {\it any} $a\in Z_M (t) \cap Z^{-1}_M (t)$ there exists  $\tilde{M}$--critical denominator $x(a)$}
\]
(with a suitable $\tilde{M}$).
To obtain a lower bound for the number of such $a$ we make the following stronger assumption.

\begin{assumption} 
Let $I \subseteq Z_M(t)$ be a $N^{1/100}$--good interval. Split $Z_M(t)$ into subintervals of length $N^{1/100}$ and let $\d\in [1/2,1]$ be a real parameter.  
We say that 
$I$ satisfies  $\d$--Assumption if
$\d$--proportion 
of such subintervals contain a point  $a\in I \cap Z^{-1}_M (t)$ such that there exists  $\tilde{M}$--critical denominator $x(a)$.
Furthermore, we assume that condition \eqref{cond:gcd} holds for this $a$ with $c_*=1/1000$, say.
\label{a:half}
\end{assumption}

If Assumption \ref{a:half} violates for a half of intervals of $Z_M(t)$ with $\d =1/2$, say, than we obtain Theorems \ref{t:Larcher}, \ref{t:main_new}, \ref{t:Zaremba} immediately (if $M$ is any number equal to $o(\sqrt{\log q})$, then we actually obtain 
a huge set of 
$a$ in Theorem \ref{t:Zaremba}, namely, the number of these $a$ exceeds the expectation). 
Indeed, let $\tilde{Z}_M(t)$ be a collection of subintervals of  $Z_M(t)$ having no $a$ with a $\tilde{M}$--critical denominator $x(a)$, $|\tilde{Z}_M(t)| \ge (1-\d)|Z_M(t)|/2 = |Z_M(t)|/4$. 
Than thanks to  Corollary \ref{c:A_A*_asymptotic} (or see the beginning of Section \ref{sec:proof}) there is an asymptotic formula for the cardinality of the set $\tilde{Z}_M(t) \cap \tilde{Z}^{-1}_M(t)$ and  all partial quotients of an arbitrary $a\in \tilde{Z}_M(t) \cap \tilde{Z}^{-1}_M(t)$ are bounded by $\max\{ M,\tilde{M}\}$.
Similarly, if the reader is not interested in the case of composite $q$, then the second part of Assumption \ref{a:half} holds trivially and can therefore be ignored.

Let us pay attention right now that  
the set $Z_M (t) \cap Z^{-1}_M (t)$  has some special properties and therefore the corresponding  collection of critical denominators $\mathcal{X}(Z_M (t) \cap Z^{-1}_M (t))$ is rather specific. 
For example, it is easy to see (or consult Section \ref{sec:proof}) that 
elements of $Z_M (t) \cap Z^{-1}_M (t)$ 
have the following simple repulsion property.
For simplicity, we consider here the case of a prime 
denominator 
$q$.

\begin{proposition}
    Let $q$ be a prime number and $a\in Z_M (t) \cap Z^{-1}_M (t)$.
    Then for {\bf any} $1\le x\le q/(4Mt)$ one has $|ax| \ge t$. 
\label{p:repulsion}
\end{proposition}
\begin{proof}
    If $|ax| < t$, then $x|ax| < q/4M$. 
    But then we know that both $x,|ax| \in [t,q/(4Mt)]$ (or see formulae \eqref{f:beg_end}, \eqref{f:beg_end_ax}) and this is a contradiction. 
$\hfill\Box$
\end{proof}

\subsection{Independency of two  critical denominators}  
\label{ssec:independency}

The main observation of this subsection is that the elements of an interval $J_{u/v}$ are quite close to each other (both in the Euclidean sense and in the continued fraction expansion sense), and hence the behavior of the various $x(a_k)$, where $a_k \in J_{u/v}$, resembles the behavior of the denominators of $x(a)$ for a single $a\in J_{u/v}$.
In particular, it can be shown that the set of critical denominators $x(a)$, $a\in J_{u/v}$ is quite  sparse and even ``independent'' in 
the sense that we are ready to define. 

\begin{definition} 
Let  $C_1,\dots, C_k \ge 1$ be some parameters and $\vec{C}:=(C_1,\dots,C_k)$. We say that elements $x_1,\dots,x_k \in \Z/q\Z$ are $\vec{C}$--independent if any equation 
\[
    \a_1 x_1 + \dots + \a_k x_k \equiv  0 \pmod q \,,
    \quad \quad 
    \mbox{where}
    \quad \quad
    \a_j \in \Z,\,\quad |\a_j| \le C_j 
\]
implies that $\a_1=\dots = \a_k=0$. 
\end{definition}

For example, in the next lemma we study $(C,C)$--independency of two critical denominators. 




\begin{lemma}
    Let $M\ge 2$, $1\le C < t/4$,  $a_1,a_2 \in J_{u/v} \subseteq Z_M (t)$, $u/v=[0;r_1,\dots,r_l]$  
    and 
    $t<x_1 \neq x_2  <q/t$ 
    be  denominators of some convergents to $a_1/q,a_2/q$.  
    Then there are no $\a,\beta \in \Z$, $\a \beta \neq 0$, $|\a|, |\beta| \le C$
    such that 
\begin{equation}\label{f:k=2}
   \a x_1 \equiv \beta x_2 \pmod q \,,
\end{equation}
    provided 
\begin{equation}\label{cond:k=2}
    N^{3/2} \le \frac{\sqrt{q}}{2(M+2)^2 C} \,.
\end{equation}
\label{l:k=2}
\end{lemma}
\begin{proof}
    Suppose that \eqref{f:k=2} takes place. 
    The assumption $C< 4t$ implies that actually this equation takes place in $\Z$. 
    By \eqref{f:u',v'} 
    we have 
\[
    \frac{a_1}{q} = [0;r_1,\dots,r_{l-1},r^{(1)}_l, \dots,r^{(1)}_{s_1}]\,,
    \quad \quad 
    \frac{a_2}{q} = [0;r_1,\dots,r_{l-1},r^{(2)}_l, \dots,r^{(2)}_{s_2}] \,.
\]
    Let $\mathcal{K}= \mathcal{K}(r_1,\dots,r_{l-1})$ and $\mathcal{K}'= \mathcal{K}(r_1,\dots,r_{l-2})$. 
    Later we will show that $l>2$, so $\mathcal{K}'>0$ is correctly defined. 
    Now thanks to the definition of the interval $J_{u/v}$, as well as formulae \eqref{f:u',v'},  \eqref{f:domino} one has 
\[
    \mathcal{K}'(M+2) \le  \mathcal{K}(M+1)+\mathcal{K}' \le t
\]
    and 
\[
    t \le \mathcal{K}+\mathcal{K}(r_1,\dots,r_l) < (M+2) \mathcal{K} < (M+1) (M+2) \mathcal{K}' \,.
\]
    Therefore, in view of the inequality  $t\ge 8M^2$, which follows from \eqref{cond:k=2} (of course, one can assume that $M$ is sufficiently small in terms of $q$), we derive 
\begin{equation}\label{tmp:K,K'}
    \frac{t}{M+2} < \mathcal{K} < \frac{t}{M+1}
    \quad 
        \mbox{ and }
    \quad 
    1<\frac{t}{(M+1)(M+2)} < \mathcal{K}' \le \frac{t}{M+2} \,.
\end{equation}
	In particular, $l>2$. 
    Using formula \eqref{f:domino} again, we get 
\begin{equation}\label{sys:*}
    x_1 = c_1 \mathcal{K} + \FF{c}_1 \mathcal{K}'
    \quad 
        \mbox{ and }
    \quad 
    x_2 = c_2 \mathcal{K} + \FF{c}_2 \mathcal{K}' \,,
\end{equation}
    where $c_1 \ge x_1/(2\mathcal{K})>1$, $c_2 \ge x_2/(2\mathcal{K})>1$, thanks to the bounds $\mathcal{K} < \frac{t}{M+1}$ and  $x_1,x_2 >t$. 
    Similarly, from \eqref{sys:*}, we see that 
\begin{equation}\label{tmp:c_12}
    1\le \FF{c}_1 \le c_1 <  x_1/\mathcal{K}
        \quad 
        \mbox{ and }
    \quad 
       1\le \FF{c}_2 \le c_2 <  x_2/\mathcal{K} \,.
\end{equation}
    Here $c_1, \FF{c}_1$ depend on $x_1$ and $c_2, \FF{c}_2$ depend on $x_2$.  
    Later we will use 
    stronger 
    lower bounds for $\FF{c}_1, \FF{c}_2$.
    Namely, one has $c_1 \le (M+1) \FF{c}_1$, $c_2 \le (M+1) \FF{c}_2$ and therefore 
\begin{equation}\label{tmp:c'_12}
    \FF{c}_1 \ge \frac{x_1}{2(M+1)\mathcal{K}}
            \quad 
        \mbox{ and }
    \quad 
    \FF{c}_2 \ge \frac{x_2}{2(M+1)\mathcal{K}} \,.
\end{equation}
    Finally, substituting \eqref{sys:*} into \eqref{f:k=2}, we obtain 
\[
    (\a c_1 - \beta c_2) \mathcal{K} = (\beta \FF{c}_2 - \a \FF{c}_1) \mathcal{K}' \,.
\]
    We know that $\mathcal{K}, \mathcal{K}'$ are coprime numbers and therefore $\mathcal{K}$ divides $\beta \FF{c}_2 - \a \FF{c}_1$.  
    But 
    using  \eqref{tmp:c_12}  
    and our assumptions on $|\a|, |\beta|$, we see that the last number does not exceed
\begin{equation}\label{tmp:C_K}
    \frac{2Cx_1}{\mathcal{K}} \le \frac{2Cq}{t\mathcal{K}} < \mathcal{K}
\end{equation}
    thanks to the bound $\mathcal{K}>\frac{t}{M+2}$ and the 
    condition 
    \eqref{cond:k=2}. 
    Hence $\beta \FF{c}_2 = \a \FF{c}_1$ and $\a c_1 = \beta c_2$.
    It follows that $c_1 \FF{c}_2 = c_2 \FF{c}_1$. 
    We know from the formula \eqref{f:domino} that $c_1,c_2, \FF{c}_1, \FF{c}_2$ are some continuants defined by 
    continuants $x_1,x_2$.
    Hence it is easy to see that  $(c_1,\FF{c}_1) = (c_2,\FF{c}_2) =1$. 
    Thus  the identity $c_1 \FF{c}_2 = c_2 \FF{c}_1$ gives us 
    $c_2=kc_1$ and $\FF{c}_2=k\FF{c}_1$ for some integer $k$. 
    But then system \eqref{sys:*} implies that $x_2 = kx_1$, thus $x_1$ divides $x_2$,  and, similarly, $x_2$ divides $x_1$.
    Thus $x_1=x_2$ and therefore relation \eqref{f:k=2} is impossible.
    This concludes the proof.  
$\hfill\Box$
\end{proof}


\begin{remark}
\label{r:coprimality_J}
 The same argument shows that 
 any two critical denominators $x_1 = x_1(a_1)$, $x_2= x_2(a_2)$ are almost coprime, that is,  
    $g:=(x_1,x_2) < 4(M+2)^2 N^2$.
    Indeed, let $x_1 = g \tilde{\a}$, $x_2 = g \tilde{\beta}$, where $(\tilde{\a}, \tilde{\beta}) = 1$. 
    Then $\tilde{\beta} x_1 = \tilde{\a} x_2$ and 
    we have that either $|\tilde{\a}|$ or $|\tilde{\beta}|$ is at least $C$, where $C>\mathcal{K}^2/2x_1$, see computations in \eqref{tmp:C_K}. 
    It follows that 
    \[
        g=(x_1,x_2) \le x_1/C \le 4 x^2_1/\mathcal{K}^2 < 4(M+2)^2 x^2_1/t^2 \le 4(M+2)^2 N^2 \,.
    \]
    Thus, the collection  of critical denominators is not just sparse or ``independent'' in the sense of Lemmas \ref{l:k=2}, \ref{l:k=3}, 
    but they are also almost coprime to each other.  This is reminiscent of the behavior of the denominators of $x(a)$ for a fixed fraction $a/q$.
\end{remark}


Let us make one remark on $\vec{C}$--dependency between three critical denominators which easily follows from Lemma \ref{l:k=2}. 
The result presented below shows that the coefficients $(\a,\beta,\gamma)$ in \eqref{f:k=3_projecive_1} are uniquely determined  by $x,y,z$, provided condition \eqref{cond:k=3_projecive} holds.  

\begin{proposition}
    Let $M\ge 2$, $C_1,C_2,C_3, C'_1, C'_2, C'_3 \ge 1$,  $a,b,c \in J_{u/v} \subseteq Z_M (t)$, $u/v=[0;r_1,\dots,r_l]$  
    and 
    $x,y,z\in [t,q/t]$
    be  denominators of some convergents to $a/q$, $b/q$ and $c/q$, correspondingly. 
    Suppose that 
\begin{equation}\label{f:k=3_projecive_1}
    \a x +\beta y + \gamma z = 0 \,, \quad (\a,\beta,\gamma)=1\,, \quad |\a| \le C_1\,, |\beta|\le C_2\,, |\gamma| \le C_3 \,,
\end{equation}
    and 
\begin{equation}\label{f:k=3_projecive_2}
    \a' x +\beta' y + \gamma' z = 0 \,, \quad (\a',\beta',\gamma')=1\,, \quad |\a'| \le C'_1\,, |\beta'|\le C'_2\,, |\gamma'| \le C'_3 \,. 
\end{equation}
    Finally,  we assume that $16 \| \vec{C}\|_\infty \| \vec{C}'\|_\infty \le t$, 
    and 
\begin{equation}\label{cond:k=3_projecive}
    N^{3/2} \le \frac{\sqrt{q}}{32 M^2 \| \vec{C}\|_\infty \| \vec{C}'\|_\infty}  \,.
\end{equation}
    Then $\a=\a'$, $\beta = \beta'$ and $\gamma = \gamma'$. 
\label{p:k=3_projecive}
\end{proposition}
\begin{proof}
    Using \eqref{f:k=3_projecive_1}, \eqref{f:k=3_projecive_2}, we obtain 
\[
    (\beta \a' - \a \beta')y +  (\gamma \a' - \a \gamma')z = 0 \,.
\]
    By Lemma \ref{l:k=2} and our condition  \eqref{cond:k=3_projecive} we see that $\beta \a' = \a \beta'$ and $\gamma \a' = \a \gamma'$. 
    The last equalities imply  $\a=\a'$, $\beta = \beta'$ and $\gamma = \gamma'$ in view of coprimality of $\a,\beta,\gamma$ and $\a',\beta',\gamma'$. 
    This completes the proof. 
$\hfill\Box$
\end{proof}

\bp

In our study of $\vec{C}$--independency between three critical denominators we use the basic 
quantity 
$d(x,y)$.
Namely, let  $x = x(a) = \mathcal{K} (r_1,r_2,\dots, r_s)$ be the denominator of a convergent to $a/q$ and $y= y(b) = \mathcal{K} (r'_1,r'_2, \dots, r'_{s'})$  be the denominator of a convergent to $b/q$.
Then we put 
\begin{equation}\label{def:d(x,y)}
    \FF{x} = \mathcal{K} (r_2,\dots, r_s) \,, 
    \quad \quad 
    \FF{y} = \mathcal{K} (r'_2, \dots, r'_{s'}) \,,
    \quad \quad \mbox{and}  \quad \quad 
    d(x,y) := x \FF{y} - y \FF{x} \,.
\end{equation}
Recall that there are two natural ways how to measure the distance between $x(a)$ and $y(b)$.
The first one is the Euclidean distance between $a$ and $b$ (or, equivalently, between $a/q$ and $b/q$). 
The second way is a non--archimedean: we say that $x=\mathcal{K} (r_1,\dots,r_{l-1}, r_l, \dots, r_s)$ is close to $y = \mathcal{K} (r_1,\dots,r_{l-1}, r'_l, \dots, r'_{s'})$, where $r_l \neq r'_l$, if $l$ is large. 
In other words, $x,y$ are close if they belong to the same fundamental interval defined by $r_1,\dots,r_{l-1}$ or, equivalently, the same cylinder defined by the initial digits $r_1,\dots,r_{l-1}$. 
Therefore, 
the modulus of the 
wedge product 
$|d(x,y)|$ in \eqref{def:d(x,y)} can be considered  as a version of this non--archimedean distance. 
The following instructive Lemma \ref{l:d(x,y)} below reflects this intuition.
For example,  formula \eqref{f:d(x,y)_bound} is  a version of the well--known fact that if two fractions are close  in the non--archimedean sense, then they  are automatically close in the Euclidean sense. 

\begin{lemma}
    Let $s,s',l$ be positive integers and 
\[
    x =x(a)=\mathcal{K} (r_1,\dots,r_{l-1}, r_l, \dots, r_s) \,, \quad 
    y = y (b) =\mathcal{K} (r_1,\dots,r_{l-1}, r'_l, \dots, r'_{s'})
\] be $\tilde{M}$--critical denominators $x, y$.
    Then 
\begin{equation}\label{f:d(x,y)_distance}
      \left| 
        \frac{|a-b|}{q}  - 
        \frac{|d(x,y)|}{x y}
    \right|
    \le \frac{1}{\tilde{M} x^2} + \frac{1}{\tilde{M} y^2} \,.
\end{equation}
    Also, for any $x,y$, assuming $r_l \neq r'_l$  and $s,s' \ge l+1$ 
    one has 
\begin{equation}\label{f:d(x,y)_bound}
    |d(x,y)| \ge \frac{x y |a-b| }{16 (r_l+r'_l) q} \,.
\end{equation}
    Finally, for arbitrary $x,y$ if $r_l \ge r'_l+2$ and $s,s' \ge l+1$, then the following holds 
\begin{equation}\label{f:d(x,y)_up_bound}
    |d(x,y)| \le \frac{8 x y r'_l |a-b|}{q} \,.
\end{equation}
\label{l:d(x,y)}
\end{lemma}
\begin{proof}
    Let us obtain estimate \eqref{f:d(x,y)_distance}. 
    As in \eqref{sys:*}, we have (see formula \eqref{f:domino})
\[
    \FF{x} = c_1 \mathcal{K}_* + \FF{c}_1 \mathcal{K}'_* \,,
    \quad \quad 
    \FF{y} = c_2 \mathcal{K}_* + \FF{c}_2 \mathcal{K}'_* \,,
\]  
    where $\mathcal{K}_* = \mathcal{K} (r_2,\dots,r_{l-1})$, $\mathcal{K}'_* = \mathcal{K} (r_2,\dots,r_{l-2})$.
    In terms of $x,y,\FF{x}, \FF{y}$ one has 
\[
    d(x,y) = x \FF{y} - y \FF{x} = (c_1 \FF{c}_2 - c_2 \FF{c}_1) (\mathcal{K} \mathcal{K}'_* - \mathcal{K}' \mathcal{K}_*) \,.
\]
    Then we 
    obtain 
\[
    \left| 
        \left| \frac{a}{q} - \frac{b}{q} \right| 
        -
        \left| \frac{c_2 \mathcal{K}_* + \FF{c}_2 \mathcal{K}'_*}{c_2 \mathcal{K} + \FF{c}_2 \mathcal{K}'} - \frac{c_1 \mathcal{K}_* + \FF{c}_1 \mathcal{K}'_*}{c_1 \mathcal{K} + \FF{c}_1 \mathcal{K}'} \right|
    \right|
    =
    \left| 
        \left| \frac{a}{q} - \frac{b}{q} \right| - 
        \frac{|c_1 \FF{c}_2 - c_2 \FF{c}_1| |\mathcal{K}' \mathcal{K}_* - \mathcal{K}'_* \mathcal{K}|}{x y}
    \right|
\]
\[
    =
      \left| 
        \left| \frac{a}{q} - \frac{b}{q} \right| - 
        \frac{|d(x,y)|}{x y}
    \right|
    =
    \left| 
        \left| \frac{a}{q} - \frac{b}{q} \right| - \left| \frac{\FF{x}}{x} - \frac{\FF{y}}{y} \right|
    \right|
\]
\[
    \le 
    \left| \frac{a}{q} - \frac{\FF{x}}{x} \right| + \left| \frac{b}{q} - \frac{\FF{y}}{y} \right| 
    \le \frac{1}{\tilde{M} x^2} + \frac{1}{\tilde{M} y^2} 
\]
    due to $\mathcal{K}' \mathcal{K}_* - \mathcal{K}'_* \mathcal{K} = \pm 1$. 
    Also, notice that 
    $|d(x,y)| = |c_1 \FF{c}_2 - c_2 \FF{c}_1|$.

    To get formula \eqref{f:d(x,y)_bound}, we need a more careful analysis. 
    Let $\o_j = \mathcal{K} (r_j, \dots, r_s)$, $\o'_j=\mathcal{K} (r'_j, \dots, r'_{s'})$, $j=l,\dots, l+3$. 
    Then $\o_l = r_l \o_{l+1} + \o_{l+2}$,
     $\o'_l = r'_l \o'_{l+1} + \o'_{l+2}$ and therefore 
\begin{equation}\label{tmp:d_omega}
    |d(x,y)| = (r_{l} - r'_{l}) \o_{l+1} \o'_{l+1}
    + \o_{l+2} \o'_{l+1} - \o_{l+1} \o'_{l+2}
    = 
    (r_{l} - r'_{l}) \o_{l+1} \o'_{l+1}
    \pm  d(\o'_{l+1}, \o_{l+1}) \,,
\end{equation}
    where we have assumed 
    without loosing of the generality that $r_{l} > r'_{l}$.
    Then if either $2\o_{l+2} \le \o_{l+1}$ or 
    $2\o'_{l+2} \le \o'_{l+1}$, then $|d(x,y)| \ge 2^{-1} \o_{l+1} \o'_{l+1}$. 
    If not, then $r_{l+1} = r'_{l+1} = 1$.
    In this case $|d(\o'_{l+1}, \o_{l+1})| = |d(\o'_{l+2}, \o_{l+2})|$ and using 
    $\o_{l+1} \ge \o_{l+2} + \o_{l+3} \ge 2 \o_{l+3}$,
    $\o'_{l+1} \ge 2 \o'_{l+3}$, we see that in any case $|d(x,y)| \ge 2^{-1} \o_{l+1} \o'_{l+1}$. 
    Thus using the notation of Lemma \ref{l:k=2} and formula \eqref{f:domino}, we derive 
\[
    |d(x,y)| \ge \frac{\o_{l+1} \o'_{l+1}}{2} \ge \frac{x y}{2(r_l+2)(r'_l+2) \mathcal{K}^2} \,.
\]
    It remains to use bounds 
\[
    \left| \frac{a}{q} - \frac{\FF{\mathcal{K}}}{\mathcal{K}} \right| \le \frac{1}{r_l \mathcal{K}^2} \,,
    \quad \quad 
    \left| \frac{b}{q} - \frac{\FF{\mathcal{K}}}{\mathcal{K}} \right| \le \frac{1}{r'_l \mathcal{K}^2} \,,
\]
    and derive 
\[
    |d(x,y)| \ge \frac{x y}{2(r_l+2)(r'_l+2) \mathcal{K}^2} 
    \ge 
    \frac{x y |a-b| ((r_l)^{-1}+(r'_l)^{-1})^{-1}}
    {2(r_l+2)(r'_l+2)q} \ge 
    \frac{x y |a-b|}{16 (r_l+r'_l) q} 
\]
    as required.

    Finally, to get \eqref{f:d(x,y)_up_bound} let us return to \eqref{tmp:d_omega}. Then we have 
\[
    |d(x,y)| \le (r_{l} - r'_{l} + 1) \o_{l+1} \o'_{l+1} 
    \le r_l \o_{l+1} \o'_{l+1} \le \frac{x y}{r'_l \mathcal{K}^2} \,.
\] 
    Since $r_l \ge r'_l+2$, it follows that the Euclidean distance between $a/q$ and $b/q$ is at least the length of the fundamental interval 
    defined by $r_1,\dots,r_{l-1}, r'_l+1$. 
    In other words, it is at least $(2(r'_l+1)^2 \mathcal{K}^2)^{-1}$. 
    Therefore, 
\[
    |d(x,y)|  \le \frac{x y}{r'_l \mathcal{K}^2} 
    \le 
    \frac{2(r'_l+1)^2 x y |a-b|}{r'_l q} 
    \le 
    \frac{8  x y r'_l |a-b|}{q}  \,.
\] 
    This completes the proof. 
$\hfill\Box$
\end{proof}

\bp

 There are general formulae of the cross-ratio type (or, alternatively, of the Euler type) 
    involving 
    the quantities $d(x,y)$.
    For example, for {\it any} $x,y,z$ one has 
\begin{equation}\label{f:d_CR}
    d(y,z)x+d(z,x)y+d(x,y)z=0 \,,
\end{equation}
    and, similarly,
\begin{equation}\label{f:d_CR_tilde}
    d(y,z) \FF{x} + d(z,x) \FF{y} +d(x,y) \FF{z}=0 \,. 
\end{equation}   
    Also, for an arbitrary $x_*,x,y,z$ the following holds  
\begin{equation}\label{f:d_CR_CR}
    d(x,x_*) d(y,z) =  d(x,y) d(x_*,z) - d(x,z) d(x_*,y)\,. 
\end{equation}  
    If $x= x(a),y=y(b), z=z(c) \sim \sqrt{q}$, say, then in  view of formula  \eqref{f:d(x,y)_distance} all quantities  $|d(y,z)|$, $|d(x,z)|$, $|d(x,y)|$ do not exceed 
    $O(\mathrm{diam}\{a,b,c\})$, 
    and therefore 
    we see that if $a,b,c$ are close to each other, then {\it there are} linear 
    relations 
    \eqref{f:d_CR} with 
    small coefficients 
    between {\it any} denominators $x,y,z$. 
    That is precisely why, when we construct a $\vec{C}$-independent triple $(x(a), y(b), z(c))$ in the next subsection, the numbers $a$, $b$, and $c$ will be situated sufficiently far apart from one another.

\subsection{Independency of three critical denominators}

Now we study 
$\vec{C}$--independency 
between three critical denominators 
from an interval $J_{u/v}$. 
There is no need for independence at a deeper level, as this would only improve some absolute constants below.
Let us underline one more time that in contrast to Lemma \ref{l:k=2} for a giving 
$a\in Z_M (t)$ {\it there are} such linear dependencies between its  continuants, e.g., $q_n (a) = A q_{m} (a) + B q_{m-1} (a)$ and, moreover, 
$q_n (a) = C q_{m_1} (a) + D q_{m_2} (a)$, where $n>m$, $n>m_1>m_2$, see formula \eqref{f:domino}.
Besides 
the coefficients $A,B,C,D$ can be 
rather 
small even if we perturb  $a$ a little bit. 
Finally, recall that there is a general formula \eqref{f:d_CR} of the cross--ratio--type which  connects $q_n(a)$ for distinct $a$ (but coefficients in this formula are 
rather 
large).
So, it is important to choose $q_{n_1} (a), q_{n_2} (b), q_{n_3} (c)$ for {\it distinct}  
and specific 
$a,b,c$ and we will take 
elements 
$a,b,c \in Z_M (t)$ to be rather far from each other
 in the sense that we defined in the previous subection. 


\bp 

We now proceed to the proof of our main technical result concerning $(C_1, C_2, C_3)$--in\-de\-pen\-den\-ce, namely Lemma \ref{l:k=3}.
Since this argument is quite lengthy, we break it down into several lemmas.
A key observation is that the coefficients in any linear relation \eqref{eq:k=3_L1} can be expressed in terms of the non--archimedean distances $|d(x,y)|, |d(x,z)|, |d(y,z)|$ and their divisors, such as $\mathcal{D} (x,y,z) = (d(x,y), d(x,z), d(y,z)) > 0$, which we call the {\it discriminant} of equation \eqref{eq:k=3_L1}.
By Proposition \ref{p:k=3_projecive} the discriminant does not depend on the coefficients of the equation, provided condition \eqref{cond:k=3_projecive} takes place. 
The 
discriminant is an important characteristic of our linear equation 
and in our first result we relate $\mathcal{D} (x,y,z)$ to  the coefficients of \eqref{eq:k=3_L1}.

\begin{lemma}
    Let $2\le M, N$ and $C_1, C_2, C_3 \ge 1$  be some parameters $\vec{C} := (C_1,C_2,C_3)$, 
    and suppose that 
\begin{equation}\label{cond:k=3_L1}
    N^{3/2} \le \frac{\sqrt{q}}{16 M^2 \| \vec{C}\|_\infty}  \,.
\end{equation}
    Also, let $a,b,c \in J_{u/v} \subseteq Z_M (t)$, 
    and 
    $x=x(a), y=y(b), z=z(c) \in [t,q/t]$ 
    be  critical denominators of some convergents to $a/q,b/q$  and $c/q$. 
    Suppose that 
\begin{equation}\label{eq:k=3_L1}
     \a x+ \beta y + \gamma z =0 \,, \quad \quad (\a,\beta,\gamma)=1 \,, 
\end{equation}
    where  $\a,\beta, \gamma \in \Z$, 
    $|\a| \le C_1, |\beta| \le C_2, |\gamma| \le C_3$ and $(\a,\beta,\gamma) \neq \vec{0}$. 
    Then 
\begin{equation}\label{f:d_xyz_L1}
    d (y,x) = \pm \mathcal{D} \gamma\,,
    \quad 
    d (x,z) = \pm \mathcal{D} \beta\,,
    \quad 
    d (z,y) = \pm \mathcal{D} \a \,,
\end{equation}
    where $\mathcal{D} = \mathcal{D} (x,y,z)$. 
\label{l:k=3_L1}
\end{lemma}
\begin{proof}
    By assumption one has   $(\a,\beta,\gamma) \neq \vec{0}$. 
    Then we can assume that in fact $\a \beta \gamma \neq 0$, 
    otherwise one can apply Lemma \ref{l:k=2}. 
    Using the notation and the argument of the proof of this lemma, we have 
    \begin{equation}\label{sys:*_3_L1}
    x = c_1 \mathcal{K} + \FF{c}_1 \mathcal{K}'\,, 
    \quad 
    y = c_2 \mathcal{K} + \FF{c}_2 \mathcal{K}'
    \,,
    \quad 
    z = c_3 \mathcal{K} + \FF{c}_3 \mathcal{K}'\,,
\end{equation}
    where $c_1 = c_1(x), \FF{c}_1 = \FF{c}_1 (x)$ and so on. 
    Substituting these formulae into \eqref{eq:k=3_L1}, we get 
\[
    (\a c_1 + \beta c_2 + \gamma c_3) \mathcal{K} = -
    (\a \FF{c}_1 + \beta \FF{c}_2 + \gamma \FF{c}_3) \mathcal{K}' \,.
\]
    As above using the assumption \eqref{cond:k=3_L1}, we obtain 
\[
    \frac{3\| \vec{C}\|_\infty x}{\mathcal{K}} \le \frac{3\| \vec{C}\|_\infty q}{t\mathcal{K}} < \mathcal{K}
\]
    and hence 
\begin{equation}\label{f:c_j_0_L1}
    \a c_1 + \beta c_2 + \gamma c_3 =
    \a \FF{c}_1 + \beta \FF{c}_2 + \gamma \FF{c}_3 = 0 \,.
\end{equation}
    By assumption 
    all three numbers $\a,\beta$ and $\gamma$ are coprime. 
    Let $g= (\beta,\gamma)$,  
    and then from \eqref{f:c_j_0_L1} we see that $g$ divides $c_1$ and $\FF{c}_1$.
    Recalling that $(c_1,\FF{c}_1) = (c_2,\FF{c}_2) = (c_3,\FF{c}_3) =1$, we derive that $g=1$.
    Similarly, we see that in fact all numbers $\a,\beta,\gamma$ are pairwise coprime. 
    Applying \eqref{f:c_j_0_L1} one more time and excluding $\a$, we obtain 
\[
    \beta (c_2 \FF{c}_1 - c_1 \FF{c}_2) = \gamma (c_1 \FF{c}_3 - c_3 \FF{c}_1) \,.
\]
    As in the proof of Lemma \ref{l:k=2}, we know that $c_2 \FF{c}_1 - c_1 \FF{c}_2 \neq 0$  and $c_1 \FF{c}_3 - c_3 \FF{c}_1 \neq 0$. 
    Then we see that $\gamma$ divides $c_2 \FF{c}_1 - c_1 \FF{c}_2 = \pm d (y,x)$ and $\beta$ divides $c_1 \FF{c}_3 - c_3 \FF{c}_1 = \pm d (x,z)$.
    Using the same argument for $\a$ and $\beta$, we get for some non--zero integer $\mathcal{D} = \mathcal{D}(\a,\beta,\gamma,x,y,z)$ that 
\begin{equation}\label{f:c_xyz_L1}
    c_2 \FF{c}_1 - c_1 \FF{c}_2 = \mathcal{D} \gamma\,,
    \quad 
    c_1 \FF{c}_3 - c_3 \FF{c}_1 = \mathcal{D} \beta\,,
    \quad 
    c_3 \FF{c}_2 - c_2 \FF{c}_3 = \mathcal{D} \a \,,
\end{equation}
    and therefore 
\begin{equation}\label{f:d_xyz}
    d (y,x) = \pm \mathcal{D} \gamma\,,
    \quad 
    d (x,z) = \pm \mathcal{D} \beta\,,
    \quad 
    d (z,y) = \pm \mathcal{D} \a \,.
\end{equation}
    Equivalently,  $\mathcal{D}  = \mathcal{D} (x,y,z) = \pm (d (y,x),d (x,z), d (z,y))$. 
    Notice also that thanks to \eqref{tmp:K,K'}, \eqref{tmp:c_12}  we automatically have 
    $
   \max\{|d(x,y)|, |d(x,z)|, |d(y,z)| \}
    \le 2q^2/\mathcal{K}^2 t^2 \le 2(M+2)^2 N^2$.
    So, if the parameter $\| \vec{C}\|_\infty$ is small enough (i.e., condition \eqref{cond:k=3_L1} 
    holds) and equation \eqref{eq:k=3_L1} takes place, then one can assume that $\| \vec{C}\|_\infty$ is even smaller than it guaranties by 
    \eqref{cond:k=3_L1}, namely, we get $\| \vec{C}\|_\infty \le 2(M+2)^2 N^2$.  
$\hfill\Box$
\end{proof}

\bp

In the proof we systematically 
use the following procedure: given 
two linear equations 
\begin{equation}\label{f:I_II}
    \a x + \beta y + \gamma z = 0
    \quad \quad 
    \mbox{and}
    \quad \quad 
    A X + B y + C z = 0
\end{equation}
such that $(\a,\beta,\gamma)=1$ and $(A,B,C)=1$, 
one can eliminate  $z$ (or $y$) and arrive to the new linear equation 
\begin{equation}\label{f:III}
    \tilde{\a} x+ \tilde{\beta} X+ \tilde{\gamma} y = 0 \,,
\end{equation}
    in which, after some work,  the new coefficients $\tilde{\a}, \tilde{\beta}, \tilde{\gamma}$ can be considered  coprime and can be expressed in terms of divisors of $\a,\beta,\gamma$ and $A,B,C$. 
    This can be viewed as a ``multiplication'' operation for equations \eqref{f:I_II}.
    Let us formulate this formally.

\begin{lemma}
    Suppose that all conditions of Lemma \ref{l:k=3_L1} take place and we have the following equations 
\begin{equation}\label{f:I_II_L2}
    \a x + \beta y + \gamma z = 0
    \quad \quad 
    \mbox{and}
    \quad \quad 
    A X + B y + C z = 0 \,,
\end{equation}
    where critical denominators $x=x(a) < X=X(a)$, $y=y(b)$, $z=z(c)$ belong to $[t,q/t]$.
    Put $\mathcal{D}_1 = \mathcal{D} (x,y,z)$ and $\mathcal{D}_2 = \mathcal{D} (X,y,z)$. 
    Then for some $g \in \Z$ one has 
\begin{equation}\label{f:a,A_L2}
    \a = \pm \frac{\mathcal{D}_2 g}{(\mathcal{D}_1, \mathcal{D}_2)} := \a_* g\,, \quad \quad 
    A = \pm \frac{\mathcal{D}_1 g}{(\mathcal{D}_1, \mathcal{D}_2)} := A_* g 
    \,,
\end{equation}
    and we have a new equation with  $\eps_1, \eps_2 = \pm 1$ and some nonzero  $k$,
    $1\le |k|\le 2\|\vec{C}\|^2_\infty \sqrt{Nq}/y$: 
\begin{equation}\label{f:y_eq_L2}
    ky = \eps_1 A_* \gamma_* X + \eps_2 \a_* C_* x \,, \quad \quad (k, A_* \gamma_*, \a_* C_*)=1 \,,
\end{equation}
    provided 
\begin{equation}\label{cond:k=3_L2}
    N^{3/2} \le \frac{\sqrt{q}}{16 M^2 \| \vec{C}\|^2_\infty}  \,.
\end{equation}
    In particular,   if $x= q_n (a) = \mathcal{K}(r_1,\dots,r_n)$ and $X = q_{n+m} (a) = \mathcal{K}(r_1,\dots,r_{n+m})$, then 
\begin{equation}\label{f:gamma_L2}
    |\gamma_*| \le \frac{y}{x} + \frac{\mathcal{D}_2 C_3 x}{X} 
    \quad \quad 
    \mbox{and}
     \quad \quad 
     |C_*| \le \frac{|C|}{|\gamma|} \cdot |\gamma_*|
    \,.
\end{equation}
\label{l:k=3_L2}
\end{lemma}
\begin{proof}
    In equations \eqref{f:I_II_L2}  we have the same numbers $y,z$ and hence the same quantity $d(y,z)$. 
    Thanks to Lemma \ref{l:k=3_L1} it follows that 
\begin{equation}\label{f:d_two_ways}
    d(y,z) = \pm \a \mathcal{D}_1 = \pm A \mathcal{D}_2 
\end{equation}
    and we obtain \eqref{f:a,A_L2}.
    In particular, we see that the discriminants $\mathcal{D}_1$, $\mathcal{D}_2$ control $\a_*$ and $A_*$.

    Similarly, using  formulae \eqref{f:d_xyz_L1} and \eqref{f:d_two_ways}, we get  
\[
    d(x,y) \a_* C_* = 
    \pm 
    d(X,y) \gamma_* A_* \,,
\]
    and therefore 
\[
    y
    \left(\a_* C_* \FF{x} \mp   \gamma_* A_* \FF{X} \right) 
    = \FF{y} 
    \left(\a_* C_* x \mp \gamma_* A_* X \right)\,.
\]
    Since $(y,\FF{y})=1$, it follows that 
\begin{equation}\label{tmp:19.06_1}
    ky =  \eps_1 A_* \gamma_* X + \eps_2 \a_* C_* x
\end{equation}
    for some nonzero  $k$,
    $1\le |k|\le 2\|\vec{C}\|^2_\infty \sqrt{Nq}/y$. 
Finally, from the definitions $(A_* \gamma_*, \a_* C_*)=1$ and hence by Lemma \ref{l:k=3_L1}, we have      $(k, A_* \gamma_*, \a_* C_*)=1$. 
In particular,  the integer $k$ divides 
\[
    |d(x,X)| = |\mathcal{K}(r_{n+2}, \dots, r_{n+m})|  \le X/x
\]
by virtue of formula \eqref{f:domino} (or see formula \eqref{tmp:23.06_1} below), which means that $k$ is an even smaller number.
    The first estimate of \eqref{f:gamma_L2} immediately follows from \eqref{tmp:19.06_1} and the bounds $|C_*| \le |C| \le C_3$ and $|\a_*| \le \frac{\mathcal{D}_2}{(\mathcal{D}_1, \mathcal{D}_2)} \le \mathcal{D}_2$. 
    The second one is a consequence of the identity $C/\gamma = C_*/\gamma_*$. 
    This completes the proof. 
$\hfill\Box$
\end{proof}

\bp 

Now we specify the choice of $x$ and $X$ in \eqref{f:I_II_L2}, namely,  we assume that $x=q_n(a) = \mathcal{K}(r_1,\dots,r_n)$ and $X=X_m = q_{n+m}(a) = \mathcal{K}(r_1,\dots,r_{n+m})$ for some $a\in J_{u/v} \subseteq Z_M (t)$, $a/q = [0;r_1, \dots, r_s]$ and a positive integer $m$. 
Let $\mathcal{D}_m = \mathcal{D} (X_m,y,z)$, so  $\mathcal{D}_0 = \mathcal{D} (x,y,z)$. 
We show (see formulae \eqref{f:k=3_L3_gcd}---\eqref{f:k=3_L3_D_small}
below) that if one controls partial quotients of $a$, then it is always possible to find a linear equation with a  small discriminant $\mathcal{D}_m$.

\begin{lemma}
    Suppose that all conditions of Lemma \ref{l:k=3_L2} take place.
    Then 
\begin{equation}\label{f:k=3_L3_gcd}
    (\mathcal{D}, \mathcal{D}_m) ~|~ \mathcal{K} (r_{n+2}, \dots, r_{n+m}) \,,
\end{equation}
    and 
    the product  $\prod_{j=1}^m \mathcal{D}_j$ divides $\a G_m$, where the number $G_m$ does not exceed 
\begin{equation}\label{f:k=3_L3_C1_small}
    G_m \le 
    \prod_{i,j =0,\, i<j}^m (\mathcal{D}_i, \mathcal{D}_j) 
    \,.
\end{equation}
    In particular, if $r_i \le M_*$ for all $i\ge n+2$, then there is $j\in [m]$ such that 
\begin{equation}\label{f:k=3_L3_D_small}
    \mathcal{D}_j \le (2M_*)^{(m+1)^2} C^{1/m}_1 \,.
\end{equation}
\label{l:k=3_L3}
\end{lemma}
\begin{proof}
    Let $g=(\mathcal{D}_m, \mathcal{D})$ and $\mathcal{K}  := \mathcal{K} (r_{n+2}, \dots, r_{n+m})$. 
    By 
    the recurrence formula, 
    we get 
\begin{equation}\label{tmp:23.06_1}
    d(X_m, y) = \mathcal{K} (r_{n+1}, \dots, r_{n+m}) d(x,y) + \mathcal{K} (r_{n+2}, \dots, r_{n+m}) d(x_1,y) \,.
\end{equation}
    By Lemma \ref{l:k=3_L1} we know that 
    $\mathcal{D}_m$ divides $d(X_m, y)$ and $\mathcal{D}$ divides $d(x, y)$.
    It follows 
    that $g$ divides $\mathcal{K} d(x_1,y)$. 
    So, it remains to check that $d(x,y)$ and  $d(x_1,y)$  are coprime. 
    The latter can be viewed in various ways.
For instance, if one assumes the contrary and uses the recurrence relation, then we see that $d((0,1),y)$ and $d((1,c),y)$ (where $c\in \Z$) share a common factor. This is impossible, since $y = y(b)$ is the critical denominator of the fraction $b/q$. 
    Thus formula \eqref{f:k=3_L3_gcd} follows.

    For any $j\in [m]$ put $\tilde{\mathcal{D}}_j = \mathcal{D}_j/(\mathcal{D}_j, \mathcal{D}_{0})$. 
    By formula \eqref{f:a,A_L2} of Lemma \ref{l:k=3_L2} we know that for any 
    $j\in [m]$ the number $\tilde{\mathcal{D}}_j$ divides $\a$.
    Hence formula \eqref{f:k=3_L3_C1_small} takes place for $m=1$. 
    For $m>1$  
    we use induction to prove the auxiliary formula: 
\begin{equation}\label{f:ind_D_tilde}
    \left(\prod_{j=1}^m  \tilde{\mathcal{D}}_j \right) ~|~ 
    \a G'_m 
     \,,
    \quad \quad  \mbox{ where } \quad \quad 
G'_m \le \prod_{i,j=1,\, i<j}^m (\mathcal{D}_i, \mathcal{D}_j) \,.
\end{equation}
     Clearly, 
     formula \eqref{f:ind_D_tilde} implies 
     \eqref{f:k=3_L3_C1_small}. 
     By induction we know that $\prod_{j=1}^m  \tilde{\mathcal{D}}_j$ divides $\a G'_m$ and $\tilde{\mathcal{D}}_{m+1}$ always divides $\a$. 
    Hence $\prod_{j=1}^{m+1} \tilde{\mathcal{D}}_j$ divides $\a G'_{m+1}$, where 
\[  
    G'_{m+1} \le G'_m \cdot (\tilde{\mathcal{D}}_{m+1}, \prod_{j=1}^m \tilde{\mathcal{D}}_j)
    \le G'_m \prod_{j=1}^m (\mathcal{D}_{m+1}, \mathcal{D}_j) 
\]
    as required.

    Finally, to obtain \eqref{f:k=3_L3_D_small} it is enough to notice  that $\mathcal{K} (r_{n+2}, \dots, r_{n+m}) \le (2M_*)^{m-1}$ and therefore $G_m \le (2M_*)^{m(m+1)^2}$. 
    Using the Dirichlet principle, we find a discriminant  $\mathcal{D}_j$ with the required property.  
   This completes the proof. 
$\hfill\Box$
\end{proof}

\bp 

Now we are ready to obtain our 
main technical 
Lemma \ref{l:k=3}.
Let us underline that the $\tilde{M}$--critical denominators $y(b), z(c)$ below are of Type I but the critical denominator $x(a)$ can be of any type (in the final argument, it will be of Type II);
moreover, we do not require that $x(a)$ be a $\tilde{M}$-critical denominator.
This appears reasonable, since one can say that if we are given the linear equation \eqref{eq:k=3}, it suffices to control only two variables -- and the equation will then automatically provide the necessary information regarding the remaining variable.

\begin{lemma}
     Let $m$  be a positive integer, $2 \le M\le M_* \le N^{1/100}$, $2\le \tilde{M}  \le N^{1/100}$, 
     $C_1, C_2, C_3 \ge 1$, 
    $H \le N^{49/100}$ 
    be some parameters, $\vec{C} := (C_1,C_2,C_3)$, 
    and suppose that 
\begin{equation}\label{cond:k=3}
    N^{3/2} \le \frac{\sqrt{q}}{16 M^2 \| \vec{C}\|^2_\infty}  \,.
\end{equation}
    Let $I \subseteq J_{u/v} \subseteq Z_M (t)$, $u/v=[0;r_1,\dots,r_l]$  be an interval, 
    satisfying $\d$--Assumption, 
    and assume that 
\begin{equation}\label{cond:H^2}
    |I| \ge \frac{20N}{\d^{60} \tilde{M}H^2} \,.
\end{equation}
    Also, 
    suppose that 
    for $q(a) \ge \sqrt{qN}/H$  all partial quotients of $a$ are bounded by $M_*$. 
    Then there exists 
    an interval $J_* \subseteq I$, 
    $|J_*|\ge \d^{60} |I|/5$, 
    $a\in J_*$, $b,c\in I$, $a<b<c$, $c-a \le 3 (b-a), 3(c-b)$ 
    and  distinct $\tilde{M}$--critical denominators $y=y(b), z=z(c)$, $\D:= \max\{y,z\} \in [tH,\sqrt{q}]$ with the following property.
    Given two critical denominators $x_{s_0} = q_{s_0} (a)$ 
    and $x_{s_*} = q_{s_*} (a)$, $s_0<s_*$
    both at least $\sqrt{Nq}/H$ 
     there is a critical denominator $x=q_s (a)$ such that any linear equation 
\begin{equation}\label{eq:k=3}
  \a x+ \beta y + \gamma z =0 \,, 
\end{equation}
 where  $\a,\beta, \gamma \in \Z$, 
 $|\a| \le C_1, |\beta| \le C_2, |\gamma| \le C_3$
 has only the trivial solution $\a=\beta=\gamma = 0$ and such that      
     either $s\in \{ s_*,\dots, s_*+m\}$, 
 or  $x=q_s (a)$, $s\in \{ s_0 ,\dots, s_0+m^2\}$ under the additional condition 
\begin{equation}\label{cond:Delta_large}
    |I| \ge  
    \frac{25 C_3 q}{\d^{60} x_{s_*} \Delta} \cdot (2M_*)^{m(m+1)^2} (C^{3}_1 C_3)^{1/2m}  \,.
\end{equation}
\label{l:k=3}
\end{lemma}
\begin{proof}
    Let $N_* := N/H^2 \ge N^{1/50}$.
    Consider the collection $\mathcal{C}$ of $\tilde{M}$--critical denominators $x\in [tH,q/t]$ for $x=x(a)$, where $a$ belongs to the interval $I$. 
    We are going to choose 
    $x = x(a)$ and $y = y(b), z=z(c)$ both in $[tH,\sqrt{q}]$ 
    in a special way and 
    we assume for concreteness that $a<b<c$. 
    Having two points $a,b$ we notice that if $\mathcal{X}(a) \cap \mathcal{X}(b) \neq \emptyset$,  then
    for some $w\in \mathcal{X}(a) \cap \mathcal{X}(b)$  one has $|aw| \le q/(\tilde{M}w)$ and $|bw| \le q/(\tilde{M}w)$. 
    Thus the distance $|b-a| \le |I| \le 16MN$ does not exceed $2q/(\tilde{M}w^2)$ (the same can be seen from formula \eqref{f:d(x,y)_distance}).
    Below we consider $w\ge tH$ only and hence the distance between $a$ and $b$ is at most
\begin{equation}\label{f:N/H^2}
    \frac{2q}{\tilde{M}w^2} \le \frac{2N}{\tilde{M}H^2} = \frac{2N_*}{\tilde{M}} := \bar{N}_* \ge \sqrt{N_*} \ge N^{1/100} \,.
\end{equation}
    Put 
    $I_* = [\bar{N}_*]$ and let 
    $I = S \dotplus I_*$.
    By assumption  the interval 
    $I$
    is $N^{1/100}$--good and hence the collection of shifts $s\in S$ such that $s+I_*$ intersects $Z^{-1}_M(t)$ has cardinality at least 
    $|S|(1-O(N^{-\kappa/2}_*)) \ge |S|/2$.
    Moreover, we know that the interval $I$ satisfies $\d$--Assumption and hence denoting by 
    $G = G(I) \subseteq S$ the collection  of shifts $s\in S$ such that for some 
    $a_s \in (s+I_*) \cap Z^{-1}_M (t)$ there is a $\tilde{M}$--critical denominator, $w_s = w(a_s)$, we see that $|G|\ge \d |S|$.  
    Now applying Lemma \ref{l:equidistributed} with 
    $N=|I|/\bar{N}_*$, 
    $A=G$,
    $\d = \d$ and $k=5$, we find an interval $J\subseteq I$ such that $J$ is $5$--equidistributed and $|J|\ge \d^{60} |I|$. 
    Split $J$ into five parts $J_1,\dots, J_5$, let $a\in J_1$, $b\in J_3$ and $c\in J_5$. 
    Then $|a-b|, |b-c| \in [|J|/5, 3|J|/5]$ and  $|a-c| \in [3|J|/5,|J|]$.
    In other words, the distances $|a-b|, |a-c|, |b-c|$ are comparable. 
    Also, one can observe that by our assumption \eqref{cond:H^2} one has 
\[
    \bar{N}_* = \frac{2N }{\tilde{M} H^2} \le  \frac{\d^{60} |I|}{5} \le \frac{|J|}{5}
\]
    and hence 
    any $x(a), y(b), z(c) \ge tH$ are distinct, see computations in \eqref{f:N/H^2}. 
    So, we can put $J_*:= J_1$.

    Further let us fix $a,b,c$ and also  $y$, $z$ (so, the specific choice of $y=y(a)$, $z=z(b) \in [tH, \sqrt{q}]$ is not particularly  important) and 
    vary the index $s$ in the continuant 
\[
    x_{s} = x_{s+s_0} (a) = \mathcal{K} (r_1,\dots, r_l, \dots, r_{s_0}, \dots, r_{s+s_0})  \ge tH \,, 
\]
    where $r_j \le M_*$ for $j>s_0$. 
    So, we assume that all triples 
    $(x_s(a), y(b), z(c))$
    are not $\vec{C}$--independent and therefore one has
\begin{equation}\label{eq:x_s}
    \a_s x_s + \beta_s y + \gamma_s z =0 \,,
\end{equation}
    where $|\a_s| \le C_1$, $|\beta_s| \le C_2$ and  $|\gamma_s| \le C_3$. 
    Put $\mathcal{D}_s = \mathcal{D} (x_s, y, z)$ and let $\a=\a_0$, $\beta=\beta_0$, $\gamma = \gamma_0$, $\mathcal{D} = \mathcal{D}_0$ for convenience.

    Using Lemma \ref{l:k=3_L3} with $x=x_0$, we find $x_{s_1}$
    such that $s_1 \le m$ and 
    \[
    \mathcal{D}_{s_1} \le (2M_*)^{(m+1)^2} C^{1/m}_1 := D \,.
    \] 
    Similarly, applying Lemma for $x=x_{s_*}$ one can find $x_{s_*+\t{s}_*}$ with $\t{s}_* \le m$ and $\mathcal{D}_{s_*+\t{s}_*} \le D$.
    Now by Lemma \ref{l:k=3_L2} with $x=x_{s_1}$ and $X=x_{s_*+\t{s}_*}$, we arrive to the equation \eqref{f:y_eq_L2}, namely, 
\begin{equation}\label{eq:y,x,X} 
    ky = \pm \a^*_{s_*+\t{s}_*} \gamma^*_{s_1} x_{s_*+\t{s}_*}  \pm \a^*_{s_1} \gamma^*_{s_*+\t{s}_*} x_{s_1} \,.
\end{equation}
    Here the numbers $\a^*_{s_*+\t{s}_*}$, $\gamma^*_{s_1}$, $\a^*_{s_1}$, $\gamma^*_{s_*+\t{s}_*}$ are defined in \eqref{f:a,A_L2}. 
    By formula \eqref{f:a,A_L2} of Lemma  \ref{l:k=3_L2}, we have 
\begin{equation}\label{tmp:alpha}
    |\a^*_{s_1}|,\,  |\a^*_{s_*+\t{s}_*}| \le D \,.
\end{equation}
    Combining  both  inequalities  of \eqref{f:gamma_L2}, 
    we derive 
\begin{equation}\label{tmp:gamma}
    |\gamma^*_{s_*+\t{s}_*}| \le \frac{4D C^2_3 x_{s_1}}{|\gamma_{s_1}| x_{s_*+\t{s}_*}} \,.
\end{equation}
    By assumption \eqref{cond:H^2} and $x,y,z\ge tH$, 
    we have in view of formulae \eqref{f:d(x,y)_distance}, \eqref{f:d_xyz_L1} that 
\begin{equation}\label{f:alpha_rho}
    \frac{yz|b-c|}{2|\mathcal{D}| q} \le |\a| = \frac{|d(z,y)|}{|\mathcal{D}|} \le  \frac{2yz |b-c|}{|\mathcal{D}| q} \,,
\end{equation}
    and similarly for $|\beta|, |\gamma|$. 
    Of course, the same is applicable to any equation \eqref{eq:x_s} and therefore 
    one we can estimate $|\gamma_{s_1}|$ from below as 
    $|\gamma_{s_1}| \ge \frac{x_{s_1} y |J|}{10Dq}$.
    Substituting the last two bounds into \eqref{tmp:gamma}, we derive 
\begin{equation}\label{tmp:20.06_1}
    |\gamma^*_{s_*+\t{s}_*}| \le \frac{40 D^2 C^2_3 q}{y x_{s_*+\t{s}_*} |J|} \,.
\end{equation}
    Applying the same argument to equation \eqref{eq:y,x,X} and using \eqref{tmp:alpha}, we obtain 
\[
    \frac{|J| x_{s_*+\t{s}_*} y}{5q} \le |d(x_{s_*+\t{s}_*},y)| \le D
    |\gamma^*_{s_*+\t{s}_*}| \cdot 
    \mathcal{D} (x_{s_1}, x_{s_*+\t{s}_*}, y) 
    \le 
     \frac{40 D^3 C^2_3 q}{yx_{s_*+\t{s}_*} |J|} \cdot \mathcal{D} (x_{s_1}, x_{s_*+\t{s}_*}, y) \,.
\]
    Recalling that $|J|\ge \d^{60} |I|$, it gives us 
\begin{equation}\label{f:final?}
    \d^{60} |I| \le |J| < \frac{25 D^{3/2} C_3 q}{y x_{s_*+\t{s}_*} } \cdot 
    \mathcal{D}^{1/2} (x_{s_1}, x_{s_*+\t{s}_*}, y) \,.
\end{equation}
    It remains to estimate the discriminant $\mathcal{D} (x_{s_1}, x_{s_*+\t{s}_*}, y)$; however, the same reasoning can be used for this purpose.
    Indeed, applying  Lemma \ref{l:k=3_L3} several times,  
    we find $s_2, \dots, s_k \le m$ such that 
    $\mathcal{D}_{s_1+s_2+\dots + s_j} \le D$ for all $j\in [k]$. 
    For each of these indices  we have the corresponding  equation \eqref{eq:y,x,X}  and  hence the corresponding bound \eqref{f:final?} with the discriminant 
    \begin{equation}\label{tmp:disc_j}
    \mathcal{D} (x_{s_1+s_2+\dots + s_j},  x_{s_*+\t{s}_*}, y) 
    \end{equation}
    instead of $\mathcal{D} (x_{s_1}, x_{s_*+\t{s}_*}, y)$. 
    But by the argument of Lemma \ref{l:k=3_L3} one of these discriminants \eqref{tmp:disc_j} does not exceed 
\[
    |\a^*_{s_1} \gamma^*_{s_*+\t{s}_*}|^{1/m} 
    (2M_*)^{m(m+1)^2} 
    \le 
    (D C_3)^{1/m} (2M_*)^{m(m+1)^2} \,.
\]
    Thus estimate  \eqref{f:final?}
     contradicts  our assumption 
\eqref{cond:Delta_large}. 
    This completes the proof. 
$\hfill\Box$
\end{proof}

\subsection{The middle interval}  
\label{ssec:mid_int}

We had deal with independency in the last subsection to be able to apply some results on   linear Diophantine approximations.
We start with a simple lemma.

\begin{lemma}
    Let $k$ be a positive integer, $T\ge 1$ be a real parameter,  $x_1,\dots,x_k \in \Z/q\Z$,
    $|x_j| \le X_j$, $j\in [k]$ and 
\begin{equation}\label{cond:linear_problem}
    (8k)^k \cdot \prod_{j=1}^k X_j \le T q^{k-1} \,.
\end{equation}
    Put 
\begin{equation}\label{fl:R_j_def}
        R_j = X^{-1}_j \left( \frac{4k \prod_{j=1}^k X_j}{T} \right)^{1/(k-1)} \,,
\end{equation}
    and 
    assume that $R_j \ge 1$ for all $j\in [k]$ and $R_1 \dots R_k \ge 2$.
    Then there are $m_j \in \Z$ not all equal zero, $|m_j| \le 2R_j$ such that either 
\begin{equation}\label{f:linear_problem_I}
    \sum_{j=1}^k m_j x_j = 0 \,,
\end{equation}
    or 
\begin{equation}\label{f:linear_problem_II}  
    0< \left|  \sum_{j=1}^k m_j x_j  \right|
    \le
    T \,.
\end{equation}
\label{l:linear_problem}
\end{lemma}
\begin{proof}
    Put $X:= R_1 X_1 = \dots = R_k X_k$.  
    In other words, by \eqref{fl:R_j_def}, we define the number $X$ as $TX^{k-1} = 4k \prod_{j=1}^k X_j$. 
    Consider $\prod_{j=1}^k (2[R_j]+1)$ sums $\sum_{j=1}^k n_j x_j$, where $n_j$ are integers such that $|n_j|\le [R_j] \le R_j$. 
    Clearly, all these sums belong to $[-kX,kX]$ and thanks to condition \eqref{cond:linear_problem} we see that this interval belongs to $[-q/2, q/2]$.
    Suppose that two of such sums coincide. 
	Then we have $\sum_{j=1}^k n_j x_j = \sum_{j=1}^k n'_j x_j$ and hence
	\[
	\sum_{j=1}^k (n_j-n'_j) x_j = \sum_{j=1}^k m_j x_j  = 0 \,,
	\]
	where the integers $m_j := n_j-n'_j$ are not all equal to zero and $|m_j|\le 2R_j$. 
	Thus we have obtained \eqref{f:linear_problem_I}. 
    Otherwise all these $\Pi := \prod_{j=1}^k (2[R_j]+1) >1$  sums are distinct. 
    Split $[-kX,kX]$ into $n=\Pi-1$ segments.
    We have 
    $R_1 \dots R_k + 1< \Pi$ and hence by the Dirichlet principle 
    there exist integers 
	$m_j$ that  are 
	not all equal  to zero, $|m_j| \le 2R_j$ such that 
	\[
    0 < \left|\sum_{j=1}^k m_j x_j \right| \le \frac{2kX+1}{n} \le \frac{4kX}{R_1\dots R_k} 
    = 
    \frac{4kX_1 \dots X_k}{X^{k-1}} =  T
    \,.
	\]
    This concludes the proof.  
$\hfill\Box$
\end{proof}

\bp 


Now we apply Lemma \ref{l:linear_problem} to the set $Z_M (t) \cap Z^{-1}_M (t)$. 
One can consider the corollary below as an improvement of Proposition \ref{p:repulsion} for a larger number of variables. 
Again, the greatest common divisor in \eqref{f:repulsion_dep2}, as well as the condition $(a,q) \le t/(4k\| \vec{C}\|_\infty)$ can be ignored for a prime $q$.

\begin{corollary}
    Let $a\in Z_M (t)$, $k\ge 2$ be a positive integer, 
    and $1\le x_j \le X_j$ be $\vec{C}$--independent numbers, where
\begin{equation}\label{c:C_j_def}
        C_j = 2X^{-1}_j \left( \frac{4k \prod_{j=1}^k X_j}{t} \right)^{1/(k-1)}
        \ge 4
\end{equation}
    and such that condition \eqref{cond:linear_problem} takes place with $T=t$.
    Then there are $m_1,\dots,m_k\in \Z$, $|m_j| \le C_j$ such that 
\begin{equation}\label{f:repulsion_dep}
    |a(\sum_{j=1}^k m_j x_j)| \ge \frac{q}{4Mt} \,.
\end{equation}
    Similarly, suppose that  $a\in Z^{-1}_M (t)$, 
    $1\le ax_j \le X_j$, 
    $(a,q) \le t/(4k\| \vec{C}\|_\infty)$, 
    $x_1,\dots, x_k \le q/t$ be $\vec{C}$--independent numbers, where $C_j$ are defined in \eqref{c:C_j_def} and  condition \eqref{cond:linear_problem} takes place with $T=t$.
    Then there are $m_1,\dots,m_k\in \Z$, $|m_j| \le C_j$ such that 
\begin{equation}\label{f:repulsion_dep2}
    |\sum_{j=1}^k m_j x_j| \ge \frac{q}{4Mt (a,q)^2} \,.
\end{equation}
\label{c:repulsion_dep}
\end{corollary}
\begin{proof}
   Let us apply 
   Lemma \ref{l:linear_problem} with $T=t$ and so the numbers $R_j = C_j/2$ are defined in \eqref{fl:R_j_def}. 
   By this lemma we find $m_j \in \Z$, $|m_j| \le 2R_j=C_j$ such that 
\[
    0<|\sum_{j=1}^k m_j x_j| \le t \,,
\]
    Since $a\in Z_M (t)$, we see that  \eqref{f:repulsion_dep} follows thanks to Lemma \ref{l:M_crit}.

    To get \eqref{f:repulsion_dep2} put $a=a^{-1}$, $x_j = ax_j$ and apply estimate \eqref{f:repulsion_dep}. 
    The multiple $(a,q)^2$ appears due to the fact that $aa^{-1} \equiv  (a,q)^2 \pmod q$ according to our definition of $a^{-1}$ for composite $q$.
    Now suppose that 
\[
    \a_1 ax_1 +\dots + \a_k ax_k \equiv 0 \pmod q \,, 
\]
    where $|\a_j|\le C_j$, $j\in [k]$.
    In other words, we assume that the numbers $ax_1,\dots, ax_k$ are  $\vec{C}$--dependent. Dividing by $(a,q)$, we obtain 
    $\a_1 a' x_1 +\dots + \a_k a' x_k \equiv 0 \pmod {q/(a,q)}$, where $a' =a/(a,q)$. 
    Now we can multiply the last identity by $(a')^{-1}$ and therefore, we get 
\[
    \a_1 x_1 +\dots + \a_k x_k \equiv 0 \pmod {q/(a,q)} \,.
\]
    Thanks to the conditions $(a,q) \le t/(4k\| \vec{C}\|_\infty)$ and  
    $x_1,\dots, x_k \le q/t$ we see that the last equation is actually in $\Z$. 
    This is a contradiction with $\vec{C}$--independency of $x_1,\dots, x_k$. 
    This completes the proof. 
$\hfill\Box$
\end{proof}

\bp 

We are ready to  obtain the main result of this section. 
Basically, we show that there are few elements $a\in Z_M(t) \cap Z^{-1}_M(t)$ 
having $\tilde{M}$--critical denominator around $\sqrt{q}$
(more precisely, we should consider a smaller set 
$(Z_M(t) \cap Z_{M_*} (tH)) \cap (Z_M(t) \cap Z_{M_*} (tH))^{-1}$). 
The point is that if there are many of them, then some triplets of their  denominators are independent, and we get too good approximations in the linear Diophantine problem \eqref{f:linear_problem_II}.
The later  is 
impossible by repulsion inequalities \eqref{f:repulsion_dep}, \eqref{f:repulsion_dep2} of Corollary \ref{c:repulsion_dep}.

\begin{proposition}
    Let $2\le M \le M_* \le N^{10^{-10}}$, 
    $\tilde{M} \le N^{1/1000}$, 
    $H = N^{9/20}$, 
    and suppose that 
\begin{equation}\label{cond:number_x}
    N \le \frac{q^{1/9}}{2^{30} M^{} M^{}_* \tilde{M}} \,.
\end{equation}
    Let $I \subseteq J_{u/v} \subseteq Z_M (t) \cap Z_{M_*} (tH)$ 
    be an interval  and assume that 
\begin{equation}\label{cond:H^2_prop}
    \frac{20N}{\d^{60} \tilde{M} H^2} \le 
    |I| \le \frac{40 N}{\d^{60} \tilde{M} H^2} \,.
\end{equation}
    Then either $I$ does not satisfy $\d$--Assumption, or there is an $a\in I \cap Z^{-1}_M (t)$ such that 
    for some $q(a) \in [\sqrt{qN}/H, \sqrt{qN}]$ one has $q(a) \ge M_*$. 
\label{p:number_x}
\end{proposition}
\begin{proof}
     Suppose that $I$ satisfies $\d$--Assumption and for all $a\in I \cap Z^{-1}_M (t)$ we have $a\in Z^{-1}_{M_*} (\sqrt{qN}/H)$. 
     In other words, we assume that for $q(a) \in [\sqrt{qN}/H, \sqrt{qN}]$ all partial quotients of $a$ are bounded by $M_*$. 
%
%
%
    Put  $x_{s_0} = q_{s_0} (a)$ and $x_{s_*} = q_{s_*} (a)$ to be the minimal and the maximal critical denominators in $[\sqrt{Nq}/H, T]$, where $T := \sqrt{q} N^{2/7}$.
    By Lemma \ref{l:k=3} there exists 
    an interval $J_* \subseteq I$, $|J_*|\ge \d^{-60} |I|/5$, $a\in J_*$, $b,c\in I$, $a<b<c$, $c-a \le 3 (b-a), 3(c-b)$ 
    and distinct $\tilde{M}$--critical denominators $y_0=y_0(b), z=z(c)$, 
    and a critical denominator $x = q_{s} (a)$ such that either  
     $s\in \{ s_*,\dots, s_*+m\}$ or 
     $s\in \{ s_0 ,\dots, s_0+m^2\}$, 
     provided additional condition \eqref{cond:Delta_large} takes place and such that equation \eqref{eq:k=3}  has no non--trivial solutions.
    We will check 
    conditions \eqref{cond:k=3},   \eqref{cond:Delta_large}, and \eqref{cond:H^2_prop} later.
    Also, we will choose the vector $\vec{C}=(C_1,C_2,C_3)$ and set the parameter $m$ to $200$. 
    Let $I_* = c-a \le |I|$ be the diameter of the set $\{a,b,c\}$. 
    We assume that  $\max\{y_0,z\} =y_0$ (the opposite case $\max\{y_0,z\} =z$ can be considered in a similar way). 
    Moreover, let   $y\ge y_0$ be the maximal critical denominator such that $y = y_n(b) \le \sqrt{q}$ and $c_{n+1} (b) y > \sqrt{q}$. 
    Define the quantity  $\t{\mathcal{M}} \le c_{n+1} (b)$ by $\t{\mathcal{M}} y =\sqrt{q}$ and let $\D = \max\{y,z\} = y$.  
    Further 
    using the fact  that all critical denominators are greater than $tH$, we have   
\begin{equation}\label{tmp:14.08_1}
    |b x| = |a x + (b-a)x| \le |a x| + I_* x 
    \le \frac{q}{\tilde{M} x_{s_0}} + I_* x
    \le 2I_* x := X_1
    \,, 
\end{equation}
    similarly, 
\begin{equation}\label{tmp:14.08_1+}
    |b z| \le \frac{q}{\tilde{M} z} + I_* z \le  2I_* z := X_3 \,,
\end{equation}
    and, finally, 
\[
    |by| \le \frac{q}{\tilde{M}  y_0}\,,\quad \quad 
     |by| \le  \frac{q}{c_{n+1} (b) y} 
    \le \frac{q}{\tilde{\mathcal{M}} y} := X_2 \,.
\]
    In deriving estimates  \eqref{tmp:14.08_1} and \eqref{tmp:14.08_1+} we have used the bound $H^2 |I_*| \tilde{M} \ge N^{}$, which 
    follows easily from condition \eqref{cond:H^2_prop} 
    and the estimate $|I_*| \ge |J_*| \ge \d^{60} |I|/5$.

    Applying the second part of Corollary \ref{c:repulsion_dep}
    with 
    $C_j \ge 4$ defined in \eqref{c:C_j_def} and $a=b$, 
    we find integers $m_1,m_2,m_3$, $|m_j| \le C_j$ such that 
\begin{equation}\label{tmp:CH_q/t}
    C_1 x + C_2 y + C_3 z \ge 
    |m_1 x + m_2 y + m_3 z| 
    \ge \frac{q}{4Mt} \,,
\end{equation}
    provided bound \eqref{cond:linear_problem} takes place,
    and $(b,q) \le t/(12\| \vec{C}\|_\infty)$. 
    We will check all  
    these conditions 
    later
    and moreover let us assume for simplicity that $(b,q)=1$.
    Now notice that according to \eqref{c:C_j_def} 
    we have 
\[
    C_1 =2 \left( \frac{12 X_2 X_3}{tX_1} \right)^{1/2} \,,
    \quad 
    C_2 =2 \left( \frac{12 X_1 X_3}{tX_2} \right)^{1/2} \,,
    \quad 
    C_3 =2 \left( \frac{12 X_1 X_2}{tX_3} \right)^{1/2} \,,
\]
    and 
    we check that all these numbers are at least 
    four 
    later.
    Returning to \eqref{tmp:CH_q/t}, we arrive to a contradiction, 
    if we have 
    three additional inequalities 
\[
    C_1 x,\, C_2 y,\, C_3 z 
    < \frac{q}{12Mt} \,.
\]
    Considering the first expression $C_1 x$ (we have exactly the same  calculations for $C_3 z$), 
    one has 
\begin{equation}\label{f:C1_bound}
    (C_1 x)^2 = 
    \frac{48 q xz}{\tilde{\mathcal{M}} yt} 
    \le \left(\frac{q}{12Mt} \right)^2 \,,
\end{equation}
    provided 
    $y \tilde{\mathcal{M}} \ge 2^{13} M^2 txz/q$. 
    Since $\D = y$, the last condition is satisfied, provided 
\begin{equation}\label{f:x_large}
    tx \le 
    \frac{\tilde{\mathcal{M}} q}{2^{14} M^2} \,,
\end{equation}
    and we will verify this  at the very end of our proof. 
    Now the condition 
\[
    (C_2 y)^2 = \frac{192 \tilde{\mathcal{M}} |I|^2 xz y^3}{tq} \le \left(\frac{q}{12Mt} \right)^2 
\]
    gives us 
\begin{equation}\label{tmp:I,H_1-}
    y^3 z 
    \le \frac{q^3}{2^{19} M^2 \tilde{\mathcal{M}} |I|^2 xt} 
\end{equation}
    Recall the definition of the number $\D$, 
    we arrive to a new assumption
\begin{equation}\label{tmp:I,H_1}
    |I|^2 \D^4 \le \frac{q^3}{2^{19} M^2 \tilde{\mathcal{M}}  xt} \,.
\end{equation}
    It remains to check conditions  \eqref{cond:linear_problem}, \eqref{cond:k=3},  \eqref{cond:Delta_large} and other less important conditions, 
    and we always use Lemma \ref{l:A+B} which gives us the following rather crude bound  $|I| \le |J_{u/v}| \le 16 MN$.
    The first 
    condition \eqref{cond:linear_problem} 
    is $(24)^3 \cdot X_1 X_2 X_3 \le q^2 t$ and thanks to our choice of $y$ it follows from 
\[
    \frac{2^{17} |I|^2 xz q}{\tilde{\mathcal{M}} y} 
    \le 
    2^{26} M^2 q^{3/2} N^{5/2} 
    \le q^2 t \,. 
\]
    So, we should have $2^{26} N^3 M^2 \le q$ and 
    the last inequality easily follows from \eqref{cond:number_x}. 

    To check \eqref{cond:k=3} one can insure that $\|\vec{C}\|_\infty = C_2$ 
    (see conditions \eqref{cond:H^2_prop} and  $\tilde{M} \le N^{1/1000}$, $H=N^{9/20}$) 
    and hence 
\begin{equation}\label{tmp:06.11_1}
    16 M^2 C^2_2 N^{3/2} = 2^4 M^2 N^{3/2}  \cdot \frac{192 \tilde{\mathcal{M}} |I|^2 xzy}{t q} 
    \le
    2^{15} M^4 N^{9/2}
    \le \sqrt{q}
\end{equation}
    and this follows from our condition \eqref{cond:number_x}. 
    Similarly, if $g:=(a,q) > t/(12\| \vec{C}\|_\infty) \gg \sqrt{q}/N^{O(1)}$, then $g$ divides $a$ and by computations in \eqref{tmp:3.09_I-} of Lemma \ref{l:coprimality} we see that the set of such $a$ has size $O(|A| N^{-\kappa c/16})$ and therefore it is negligible. 
    Moreover, we have assumed condition \eqref{cond:gcd} 
    and then the same lemma shows that the set of $a$ where \eqref{cond:gcd}  fails is also negligible. 
    This will change the final calculations below slightly in case $(b,q), (c,q) > 1$.
    Now let us check that $C_2 \ge C_3\ge 4$ and we postpone the inequality $C_1\ge 4$ at the very end of the proof.  
    As for the condition $C_2 \ge 4$, it follows from $|I|^2 xzy \ge |I|^2 t^3 H^3\ge tq$ and the last inequality takes place due to $H\ge N^{1/3}$. 
    The bound $C_3 \ge 4$ is a consequence of  $xq\ge ytz$ and this last inequality is obvious.

    Finally, if $x=q_s (s)$ and  $s\in \{ s_0 ,\dots, s_0+m^2\}$ 
    it remains 
    to check condition \eqref{cond:Delta_large}, namely, 
\[
    |I| \ge  
        \frac{25 C_3 q}{\d^{60} x_{s_*} \Delta} \cdot (2M_*)^{m(m+1)^2} (C^{3}_1 C_3)^{1/2m}  \,.
\]
    Since $C_1\le \frac{q}{12Mtx}$, $C_3 \le \frac{q}{12Mtz}$ and $x\ge \sqrt{Nq}/H$, $z\ge tH$, we see that $C^3_1 C_3 < H^3 \cdot N/H = NH^2 < N^{2}$ and hence taking $m=200$, we obtain
\begin{equation}\label{tmp:m^3_small}
    (2M_*)^{m(m+1)^2} \cdot (C^{3}_1 C_3)^{1/2m}  
    \le 
    (2M_*)^{m(m+1)^2} \cdot N^{1/200}
    \le 
    N^{1/1000} \cdot N^{1/200}
    < N^{1/100} \,.
\end{equation}
    Similarly,  in the case $x=q_s (s)$ and $s\in \{ s_0 ,\dots, s_0+m^2\}$ one has $x\le \sqrt{qN}/H \cdot N^{1/1000}$.
    Applying the last estimate, assumption \eqref{cond:H^2_prop} and 
    the identity 
    $\t{\mathcal{M}} y =\sqrt{q}$, 
    we obtain 
\[
    \frac{625 C^2_3 q^2}{\d^{120} x^2_{s_*} \Delta^2} \cdot (2M_*)^{2m(m+1)^2} (C^{3}_1 C_3)^{1/m} 
    \le 
        \frac{2^{15} q^3 x}{\d^{120} \tilde{\mathcal{M}} tz x^2_{s_*} \Delta^3} 
        \cdot 
    N^{1/80}
    \le 
\frac{2^{15} N^{} q^{5/2}}{\d^{120}  H z x^2_{s_*} \Delta^2} 
        \cdot 
    N^{1/70}
\]
\[
    \le 
    \left( \frac{N}{\d^{60} \tilde{M} H^2} \right)^2 
    \le 
    |I|^2
\]
    or, in other words, 
\[
    2^{15} \t{M}^2 H^3 q^{5/2} \cdot  N^{1/70} 
    \le H^3 N^{-1/2} x^2_{s_*} q^{3/2}
    \le N^{} z x^2_{s_*} \D^2   
\]
    and this is obviously true thanks to our choice of 
    $T= \sqrt{q} N^{2/7}$. 
    Returning to \eqref{tmp:I,H_1}, let us check that 
\[
    |I|^2 \le \left(\frac{40 N}{\d^{60}\tilde{M} H^2} \right)^2 \le \frac{q^3}{2^{19} M^2 \tilde{M} xt \D^4}
\]
    or, in other words, thanks to the bound $x\le (2M_*)^{(m+1)^2+1} T < N^{1/1000} T$, we need to verify 
\[
    2^{30} \d^{-120} M^2 N^2 t \D^4 x
    \le 
    2^{30} \d^{-120} M^2  N^2 t \D^4 T \cdot N^{1/1000} 
    \le \tilde{M} q^3 H^4 \,.
\]
    One can check that our choice of $T= \sqrt{q} N^{2/7}$ 
    and the assumptions $2\le M \le M_* \le N^{10^{-10}}$
    imply  the last bound holds for  $H = N^{9/20}$. 
    The last inequality we need to check is $C_1\ge 4$,  and it follows from
\[
    qz \ge qtH \ge 
    tx \sqrt{q} = \tilde{\mathcal{M}} txy \,,
\] 
    and hence we need to have  (see computations in \eqref{tmp:m^3_small}) 
\[
    x \le (2M_*)^{(m+1)^2 m+1} T 
    \le N^{1/1000} T
    \le H\sqrt{q} 
\]
    and this is 
    trivially 
    true, thanks to $H=N^{9/20}$ and $T = \sqrt{q} N^{2/7}$. 
    In particular, we have $x\le q/t$ and moreover condition \eqref{f:x_large} takes place.  
    This concludes the proof.  
$\hfill\Box$
\end{proof}

\section{The proofs of Theorems \ref{t:Larcher}, \ref{t:main_new} and \ref{t:Zaremba}}
\label{sec:proof}

The proofs of Theorems \ref{t:Larcher}, \ref{t:main_new} and \ref{t:Zaremba} are carried out according to similar schemes and differ only in the choice of parameters. Let us start with Theorem \ref{t:Larcher}, 
but first of all, we recall the proof of Theorem \ref{t:main} from \cite{MMS_Korobov}, since we will use some parts of the argument, as well as the notation. 
We also simplify some steps of the proof of Theorem \ref{t:main} by using the exact asymptotic formula from Corollary \ref{c:A_A*_asymptotic}.

\subsection{Previous argument}
\label{ssec:previous}

Take a parameter $\eps \in (0,1/2]$, which we will choose later and let $N=q^{2\eps}$ and $t=q^{1/2-\eps} = \sqrt{q/N}$. 
For simplicity, the reader may assume that $q$ is a prime number.
We definitely assume that 
$t=o(\sqrt{q})$, 
$q\to \infty$ and hence we have  the condition 
\begin{equation}\label{cond:eps_p}
    \eps \gg  \frac{1}{\log q} \,.
\end{equation}
Put $A = A_M := Z_M (t)$. 
Thus, our set $A$ depends on two parameters: $\eps$ and $M$ (and the parameter $N$ depends on $\eps$ only). 
Thanks to Lemma \ref{l:A+B} we know that the set $A$ is a collection of disjoint intervals of lengths from  $N$ to 
$16MN$. 
Also, one has  
\[
    q^{w_M} N^{1-w_M} \ll TN \le |A| \ll MTN \ll M q^{w_M} N^{1-w_M}\,, 
\]
where $T\sim t^{2w_M}$. 
Now consider the set $Z_* = A\cap A^{-1}$ (recall that for composite $q$ the set $A^{-1}$ is the collection $\{ a\in [q] ~:~(a/(a,q))^{-1}/(q/(a,q)) \in A\}$, where $(a/(a,q))^{-1}$ is the inverse element modulo $q/(a,q) \ge t$.
Note that $Z_*$ is symmetric in the sense that $Z^{-1}_* = Z_*$. 
By Corollary \ref{c:A_A*_asymptotic} in the form of Lemma \ref{l:coprimality} (recall that we want to have an $a$ coprime to $q$)
the following holds 
\begin{equation}\label{f:Z_*_as_f}
    |Z_*| \ge \frac{|A|^2 \_phi (q)}{q^2} - O(|A| N^{-\kappa}) 
    \ge 
    \frac{|A|^2 \_phi (q)}{2q^2} 
    \,, 
\end{equation}
provided $|A|N^\kappa \gg q\cdot q/\_phi(q)$.
In other words, we must have
\begin{equation}\label{f:eq_J_cond_q}
    q^{2\eps(1+\kappa -w_M)} \gg q^{1-w_M}  \cdot q/\_phi(q)
\end{equation}
or, equivalently, (recall that $1-w_M \sim  1/M$) 
\begin{equation}\label{cond:eps}
    \eps \gg \frac{1}{M} \,,
\end{equation}
    as well as $\eps \gg \frac{\log \log \log q}{\log q}$ (we will see that this condition is weaker than \eqref{cond:eps}). 
Thus under the condition \eqref{cond:eps} the set $Z_*$ is nonempty and   we see that there are $z_1,z_2\in A$ with $z_1z_2 \equiv 1 \pmod q$.  
Put $a=z_1$.
By Lemma \ref{l:M_crit} we have that for all $x\le t$ and $1\le |y| < q$ with $ax\equiv y \pmod q$ the following relation $x|y| \ge q/4M$ holds.
Now let us recall the well--known fact that continued fractions are related to the question of finding the reciprocal fraction $a^{-1}$ modulo $q$, see \cite{Hinchin}. 
This formulae (alternatively, a direct application of Lemma \ref{l:M_crit}) allows us to transform  denominators of Type I to denominators of Type II. 
More precisely, we have 
\begin{equation}\label{f:inverse1}
\frac{a^{-1}}{q}=\left[0 ; c_{s}, c_{s-1} \ldots, c_{1}\right]
\quad \quad \quad \quad 
\text {if }  s \text { is even }
\end{equation}
and 
\begin{equation}\label{f:inverse2} 
\frac{a^{-1}}{q}=\left[0 ; 1, c_{s}-1, c_{s-1} \ldots, c_{1}\right] \quad \text { if } s \text { is odd. }
\end{equation} 
Thus 
in view of formulae \eqref{f:inverse1}, \eqref{f:inverse2} for any  $x\le t$ and $1\le |y| <q$ with $a^{-1} x\equiv y \pmod q$ one has $x|y| \ge q/4M := q_M$.
The last modular equation is equivalent to 
$x \equiv ya \pmod q$ and hence any solution of \eqref{eq:M_crit} 
satisfy  
\begin{equation}\label{f:beg_end}
    x|y| = x|ax| \ge q_M 
    \quad \quad 
    \mbox{for} \quad \quad x\in [t] 
    \quad \quad \mbox{ and } \quad \quad 
    x\in  \left[ q_M/t, q \right) \,.
\end{equation}
Now suppose that there are $x,y \in [t,q_M/t]$ such that $x|y| = x|ax|< q_M$.  
Recall that we call such denominators
$x=x(a)$ 
as {\it critical denominators}. 
Notice that since $Z_* = Z^{-1}_*$ 
one can assume 
that either $x$ or $y$ less than $\sqrt{q}$, see formula \eqref{f:domino}. Moreover,  if, say, $x\le y$, then from $x|y|< q_M$, we derive $x<\sqrt{q_M} < \sqrt{q/M}$.   
After that we chose  $M\sim \frac{\log q}{\log \log q}$ in \cite{MMS_Korobov} and after some computations we came to a contradiction,  
since $x<\sqrt{q/M}$, but at the same time,  
$x\ge t$ (for details, see \cite{MMS_Korobov}).

Finally, let us give a rough upper bound for the partial quotients of elements of $Z_*$, which is, firstly, very instructive, and secondly, can in principle be used in an alternative proof of Proposition \ref{p:number_x}, replacing the argument with the quantity $M_*$. 
For simplicity, assume that $q$ is a prime number. 
Suppose that $a\in Z^{-1}_M (t)$ and then putting $x=|ax|$ in \eqref{f:beg_end}, we get 
\begin{equation}\label{f:beg_end_ax}
    |x| |ax| = |ax| |a^{-1} (ax)|  \ge q_M 
    \quad \quad 
    \mbox{for} \quad \quad |ax| \in [t] 
    \quad \quad \mbox{ and } \quad \quad 
    |ax| \in  \left[ q_M/t, q \right) \,.
\end{equation}
Therefore, if we assume that for $x\in [t,q_M/t]$ one has $x|ax| \le q/G$, where $G\ge 4M$, then
\begin{equation}\label{f:G_bound}
    t^2 \le x|ax| \le q/G\,.
\end{equation}
    Thus by Lemma \ref{l:M_crit} all partial quotients of $a$ are bounded by $G = \max\{ N, 4M\}$. 
    One can easily refine the last bound removing some elements from $Z_M(t)$ and modifying condition \eqref{cond:eps} slightly. 
    Indeed, consider the dyadic decomposition $A = \bigcup_{\D_1,\D_2} A_{\D_1, \D_2}$, where $\D_1,\D_2 \in [t, q_M/t]$ and 
\[
    A_{\D_1, \D_2} = \{ a\in A ~:~ \D_1 \le x \le 2\D_1,\, \D_2 \le |ax| \le 2\D_2,\, x|ax| < q_M \} \,.
\]
    Then consider the set $A' = \bigcup_{\D_1,\D_2 ~:~ \D_1 \D_2 < qN^{-c}} A_{\D_1, \D_2}$, where $c\in (0,1)$ is an arbitrary constant.
    Each set $A_{\D_1, \D_2}$ belongs to $\pm [\D_2, 2\D_2]/[\D_1, 2\D_1]$ and hence $|A_{\D_1, \D_2}| \le 8 \D_1 \D_2$.
    Thus  thanks to our main condition \eqref{cond:eps} (now $\eps \gg_c M^{-1}$ of course), we get 
\[
    |A'| \ll q N^{-c} \cdot \log^2 N \le |A| N^{-c/2} \,.
\]
    Therefore, the number of such elements $a$ in $A$ is negligible and for the rest we have 
\[
    q N^{-c} \le \D_1 \D_2 \le x|ax| \le q/G \,.
\]
    In other words, 
\begin{equation}\label{f:rough_G}
    G = \max\{ N^c, 4M\} \,.
\end{equation}
    Thus, in particular,  we see that our set $Z_M (t) \cap Z^{-1}_M (t)$ is far from being random (we have already seen this in Proposition \ref{p:repulsion}).

\subsection{The proof of Theorem \ref{t:Larcher}}
\label{ssec:Larcher}

To 
obtain 
Theorem \ref{t:Larcher} 
we choose 
\begin{equation}\label{f:choice_M}
    M \sim \sqrt{\log \log q} 
\end{equation}
and assume that conditions \eqref{cond:eps_p}, \eqref{cond:eps} are satisfied. In particular, we have the estimate \eqref{f:Z_*_as_f} and therefore the set $Z_*$ is nonempty. 
In addition to \eqref{cond:eps_p} and \eqref{cond:eps} we also assume that 
\begin{equation}\label{cond:I&M}
    N=q^{2\eps} \gg \left(\max \left\{ M, \frac{q}{\_phi(q)}
    \right\}  \right)^C \,,
\end{equation}
    where $C>1$ is a sufficiently large but absolute constant.
    Obviously, the condition $q^\eps \gg q/\_phi(q)$ implies \eqref{cond:eps_p}, and therefore the latter can be freely ignored.

\bp

Our next aim is to delete some elements from $Z_*$ and construct a smaller set $Z'_* \subseteq Z_*$ such that still $|Z'_*| = \Omega (|Z_*|)$ and for any $a\in Z'_*$ all partial quotients $c_j (a)$ up to $tH$, $H=N^{9/20}$ are bounded by 
\begin{equation}\label{f:choice_M*}
    M_* 
    \ll M^{-1} \log q 
    \ll \log q \,. 
\end{equation}
    Of course, we also assume that $M_*\ge M$. 
 As we know from Lemma \ref{l:A+B} 
    the set  $A_M$ is a Cantor--type set.
    More precisely, it is the intersection of the progression $R_q$
    with the  Cantor--type set $Q_M(t)$ 
    which 
    is a disjoint union of some intervals $J_{u/v}$ of length at least $N$ and these intervals are indexed by the set $\overline{Q_M(t)}$.
    Now to reach the required set $A' = A'_M \subseteq A_M$ and then the set $Z'_* = A'\cap (A')^{-1}$ we successfully  delete some intervals from the set $Q_M(t) = \bigsqcup_{u,v} J_{u/v}$.
    We can remove all
    $N^{1/100}$--bad 
    intervals $J_{u/v}$  and by 
    Lemma \ref{l:dense_subint} we delete at most $qN^{-\kappa /200}$ points, so it is negligible.
    Further, taking an interval $J_{u/v}$, $\frac{u}{v}=[0;r_1,\dots,r_l] = \frac{p_l}{q_l}$ 
    we need to remove all points of the form  $[0;r_1,\dots,r_l,r_{l+1}, \dots]$, where for some $j\ge l+1$ one has $r_{j} >M_*$. 
    In other words, we need to remove all fundamental intervals 
    $$I(r_1,\dots,r_l, r_{l+1}, \dots, r_s) = \{ \a \in J_{u/v} ~:~ \a = [0;r_1,\dots,r_l, r_{l+1}, \dots, r_s]\}$$ such that for  some $j\ge l+1$ one has $r_{j} >M_*$. 
    Recall 
    that $v<v',v''<(M+1)v$ and $v^{-2}<|J_{u/v}|<2(M+1) v^{-2}$, see formula \eqref{f:J_uv}. 
    From the theory of the continued fractions  \cite[Chapter III, Sections 12, 13]{Hinchin} we see that 
     at each step we delete at most 
\begin{equation}\label{f:T_portion_up}
    \sum_{r> M_*} \frac{2}{3r^2 q_l (q_l+q_{l-1})} \le \frac{2}{3 M_* v^2} \le \frac{2 |J_{u/v}|}{3 M_*} 
\end{equation}
    points. 
    Further, any fraction $a/q$ from $Z_*$ has partial quotients bounded by 
    $O(M)$
    with, possibly,
    $4 \log (q_M/ t^2)$ 
    exceptions in 
    the interval $[t,q_M/t]$
    and thus if we take 
\begin{equation}\label{cond:T_lower}
    M_* \gg \o^{-1} M^{-1} \log q \ge 8 \o^{-1} \eps \log q \ge 4 \o^{-1}\log (q_M/t^2) 
    \,,
\end{equation}
    then it guaranties, in particular, that we arrive to a set $A'$ of size  
\begin{equation}\label{f:large_A'}
    |A'| \ge (1-\omega)|A_M| \,.
\end{equation}
    Here $\o \in (0,1)$ is any absolute constant, say, $\o = 1/32$. 
    Nevertheless, the set $A'$ is rather unstructured and 
    to have $|Z'_*| \gg |Z_*|$ we consider just subintervals $I(r_1,\dots,r_l, r_{l+1}, \dots, r_s) \subseteq J_{u/v}$, $v\ge t$ of size at least $h:=N/H^2 = N^{1/10}$.
    Also, we impose the assumption 
\begin{equation}\label{cond:J_uv}
    h \gg M_* \,. 
\end{equation}
    Then as in \eqref{f:Z_*_as_f}, we see that 
\begin{equation}\label{tmp:Z_*_N_to_h}
    |Z'_*| \ge \frac{|A'|^2 \_phi (q)}{q^2} - O(|A| h^{-\kappa}) 
    \gg 
    \frac{|A|^2 \_phi (q)}{q^2} 
\end{equation}
    thanks to our basic inequality \eqref{f:eq_J_cond_q}.
    Also, condition \eqref{cond:J_uv} is satisfied, see formulae \eqref{f:choice_M} and \eqref{f:choice_M*}. 
    Notice that to arrive to the set $Z'_*$ one can use even simpler argument which nevertheless depends on the choice of the parameter $M$ in a crucial way. 
    Indeed, by lower bound \eqref{f:AD_lower_no_Rq}  of Lemma \ref{l:AD}   we have a set $A' \subseteq A$ such that all partial quotients of $a\in A'$ are bounded by $M$, provided the denominator of $a/q$ does not exceed $tH$ and the size of $A'$ is at least
\begin{equation}\label{f:A_M/M_*}
    M^{-1} t^{2w_M} N^{w_M} (N/H^2)^{1-w_M} = M^{-1} t^{2w_M} N H^{2(w_M-1)} \,.
\end{equation}
    To have estimate \eqref{f:large_A'} we need $H^{2(w_M-1)} \gg 1$ and hence $M\gg \sqrt{\log q}$ (one can check that, actually, the term $M^{-1}$ in \eqref{f:A_M/M_*} is acceptable). 
    Of course, we have just condition \eqref{f:choice_M} but below our choice of $M$ is \eqref{f:M_Z_new} and therefore this alternative argument works. 
    Anyway, with some abuse of the notation, below we write $Z_*$ instead of  $Z'_*$.

    It remains to remove  from 
    $Z_*$ 
    the subintervals $J_{u/v}(r_1,\dots,r_l, r_{l+1}, \dots, r_s)$ that have  sizes less than 
    $h=N^{1/10}$
    but if we apply the previous 
    procedure,
    then, of course, this may turn out to be too wasteful.
    Let us remark that since 
    $qv^{-2} <|J_{u/v}| < 2(M+1)q v^{-2}$ thanks to \eqref{f:J_uv} (here we consider $J_{u/v}$ as an interval of $\Z/q\Z$), 
    it follows that 
    $x = x(a) =v> \sqrt{q} h^{-1/2} = tH$.
    Similarly, if $a\in Z_* \cap Z^{-1}_*$, then $x(a) \le q_M/tH < \sqrt{qN}/H$. 
    By Proposition  \ref{p:number_x} applied with the parameters $M,M_*$, 
    $\tilde{M} = 10 \d^{-60} \le 2^{70}$, $c_* =1/1000$
    and $H=N^{9/20}$, we see that  
    subintervals of $J_{u/v}$ of lengths, say,  
    $\frac{40h}{\d^{60} \tilde{M}}$
    cannot 
    satisfy $\d$--Assumption.    
    Of course, it is necessary to assume that condition \eqref{cond:number_x} is satisfied.
    So, all partial quotients of $A'$ are bounded by $4\max\{M,M_*,\tilde{M}\}$. 
    Now take $C=10^{10}$ in \eqref{cond:I&M} and observe that  condition
    \eqref{cond:I&M} is satisfied  
    by 
    our choice of the parameter $M$ due to a similar  bound 
\begin{equation}\label{cond:I&M_h}
    h^{} 
    \gg \max \{ M, q/\_phi(q)\}^C \,.
\end{equation}
    Indeed, the later one is follows from 
\begin{equation}\label{f:M_bad}
    M \ll 
    \frac{\log q}{\log \log q} \,.
\end{equation}
    Finally, to apply Proposition \ref{p:number_x} we need to check condition \eqref{cond:number_x}.
    This is equivalent to the inequality 
\begin{equation}\label{cond:eps_upper}
    \eps < \frac{1}{18} \,, 
\end{equation}    
    and since $\eps \sim 1/M$, then, taking a sufficiently large $M$, we can always assume that the last estimate takes place.
    Now let $\tilde{A} \subseteq A$ be the union of all points from $A$, containing in the remaining  subintervals.
    Then $|\tilde{A}| \gg |A|$ and thus we can define $Z'_* := \tilde{A}\cap \tilde{A}^{-1} \subseteq Z_*$. 
    By the same computations as in  \eqref{f:Z_*_as_f}, \eqref{tmp:Z_*_N_to_h}, we see that $|Z'_*| \gg |Z_*|$. 

\bp 

Finally, to prove Theorem \ref{t:Larcher} it remains to make the last step and we follow the method from \cite{Larcher} and \cite[proof of Theorem 2]{MU_Larcher}.
Namely, above we constructed a set $Z'_* \subseteq Z_*$ such that for all $a\in Z'_*$ and an arbitrary $x\notin [t,q_M/t]$ the equation $ax \equiv y \pmod q$ implies $x|y| \gg q/M$. 
In other words, thanks to Lemma \ref{l:M_crit} all partial quotients $c_j$ of $a/q=[c_1,\dots,c_s]$ are bounded by $O(M)$ excluding the segment when $q_j (a) \in [t,q_M/t]$.
Given $a\in Z'_*$ denote by $S'(a)$ the sum of $c_j (a)$ such that $q_j (a)\in [t,q_M/t]$ and by $S''(a)$ the sum of $c_j (a)$ such that $q_j (a)\in [t,tH]$.  
We know that $c_j (a)\le \max\{M_*, \tilde{M}\} := \bar{M}_*$ for any $j\in [s]$. 
Then by the Abel summation (or see the details in \cite{Larcher}, \cite{MU_Larcher}) 
we have  that 
\[
    \sigma:= \sum_{a\in Z'_*} S'(a) \ll \sum_{2\le c \le  \bar{M}_*} 
    \left| \{ (a,x) \in Z'_* \times [t,q_M/t] ~:~ \| ax/q\| < 1/(cx) \} \right| \,.
\]
    Splitting the sets $[\bar{M}_*]$ and $[t,q_M/t]$ in dyadic way, we decompose the sum $\sigma$ into $O(\log \bar{M}_*)$ and $O(\log (q_M/t^2))$ subsums. 
    Fixing a pair of such 
    subsums 
    as well as  an element $a\in Z'_*$,  it remains to count the number of $j$ such that $\Delta_1 \le q_j (a) < 2\Delta_1$, $\Delta_2 \le c_j(a) <2\Delta_2$ and such that 
    $\| ax/q\| < 1/\Delta_1 \Delta_2$. 
    Here $\D_1 \in [t,q_M/t]$ and $\Delta_2 \in [2,c]$ are some numbers. 
    But clearly the number of $q_j (a)$ in such a segment is at most two and hence the number of $c_j (a)$ is at most four (see details in \cite{MU_Larcher}).     
    Thus we obtain 
\begin{equation}
    \sigma \ll |Z'_*| \cdot \log \bar{M}_* \cdot \log (q_M/t^2) \,.
\end{equation}
    Another way to see the same is just to fix an interval $I := J_{u/v} \subseteq A_M$, 
    and repeat the original argument of \cite{Larcher} and \cite[proof of Theorem 2]{MU_Larcher}. 
    Namely,  as above the set $\{(a,x)\in I \times [t,tH] ~:~ \| ax/q\| < 1/(cx) \}$ has size $O(c^{-1}|I| \log (q_M/t^2))$ and after integration over $c$ we obtain $O(|I| \log \ov{M}_* \cdot \log (q_M/t^2))$.
    In other words, we have found a set $A''\subseteq A_M$, $|A''| \gg |A_M|$ such that for all $a\in A''$ one has $S''(a) = O(\log \ov{M}_* \cdot \log (q_M/t^2))$. 
    After that we can apply the argument from lines  \eqref{f:large_A'}---\eqref{cond:eps_upper}, since  
    our new set $A''$ 
    preserves  the interval structure. 
    Therefore 
    there is $a\in Z'_*$ such that the sum of its {\it all} partial quotients does not exceed (recall that  $M \sim \sqrt{\log \log q}$ and $M_* \ll \log q$, see formulae \eqref{f:choice_M}, \eqref{f:choice_M*}) 
    the quantity 
\[
    \log \bar{M}_* \cdot \log (q_M/t^2) + M \log q
    \ll 
    M^{-1} \log M_* \cdot \log q + M \log q
    \ll 
    (M^{-1} \log \log q  + M) \log q
\]
\[
    \ll 
    \sqrt{\log \log q} \cdot \log q
\]
    as required. 
    It is easy to check that bound \eqref{f:Larcher_number} follows from our construction if one puts $\d = 1/2$.
    The additional term 
$
\exp \left( \Omega \left(\frac{\log q}{\log \log q} \right) \right)
$ 
is $N^{1-w_M} = q^{\Omega(M^{-2})}$.
This completes the proof. 
$\hfill\Box$

\subsection{The proof of Theorem \ref{t:main_new}}
\label{ssec:main_new}

To obtain Theorem \ref{t:main_new} 
we 
take 
\begin{equation}\label{f:M_Z_new}
    M \sim \sqrt{\log q} = O \left(\frac{\log q}{\log \log q} \right) 
\end{equation}
to satisfy condition \eqref{f:M_bad} and hence conditions  
\eqref{cond:eps_p} and \eqref{cond:I&M}.
Similarly,      condition \eqref{cond:J_uv} 
holds. 
As always we assume that the inequality \eqref{cond:eps} takes place.
By the previous argument we have constructed the set $Z_*$ such that for any $a\in Z_*$ one has $c_j (a) =O(M)$, provided $q_j (a) \notin [t,q_M/t]$ and $c_j (a) =O(\bar{M}_*) = O(M^{-1} \log q)$ for 
$q_j(a) \in [t,q_M/t]$,
see formula \eqref{f:choice_M*}.
Hence the choice $M \sim \sqrt{\log q}$ leads to  (see Remark \ref{r:M_crit})
\begin{equation}\label{f:T2_maxs}
    \max\{ M+2, M_*\} = 
    O(\sqrt{\log q}) 
\,.
\end{equation}
Now by \eqref{f:main_M_new} we choose $M = C \sqrt{\log q}$, where $C>1$ is a sufficiently large absolute constant. 
Taking $C$ large enough we can make the maximum in \eqref{f:T2_maxs} to be $M+2$.
In other words, 
all partial quotients are bounded by $M+2$. 
As above bound \eqref{f:main_new_number} follows from the construction and our choice of $\d = 1/2$. 
This completes the proof of the theorem. 
$\hfill\Box$


\begin{remark}
\label{r:optimality}
    If one is interested in obtaining optimal bounds for the Zaremba numbers $a$, then 
    we expect 
    that the limit of our method is 
    $M\sim \sqrt{\log q}$ 
    and hence the upper bound in \eqref{f:main_M_new} is close to the optimal if one uses this circle of ideas. 
    For simplicity, let us assume that $q=p$ is a prime number.
    In the notation above 
    $|A| \gg p^{w_M} N^{1-w_M}$ 
    thanks to Lemma \ref{l:A+B} (also, see estimate \eqref{f:AD_lower_no_Rq} of Lemma \ref{l:AD}) and  therefore 
    \[
        |Z_*| \gg \frac{|A|^2}{p} \gg p^{2w_M-1} N^{2(1-w_M)} \,.
    \]
    On the other hand, using  the average argument one can easily see that the expectation of the size of the set $Z_*$ cannot be greater than $p^{2w_M-1}$ (consult paper \cite{hensley1989distribution} for sharp results). 
    Hence we must have 
    $N^{2(1-w_M)} = O(1)$. 
    Recalling that $1-w_M \sim 1/M$ and $N\sim p^{\eps} \sim p^{\Omega(1/M)}$, we obtain 
\[
    p^{\Omega(1/M^2)} \ll
    1 
\]
    and therefore 
    $M \gg \sqrt{\log p}$. 
\end{remark}


\begin{problem}
\label{pr:correct_a}
Let $q$ be a sufficiently large positive integer and 
\begin{equation*}
        1 \ll M
    = o (\sqrt{\log q}) 
\,.
\end{equation*}
    Prove that there are $\Omega(q^{2w_M-1-o(1)})$  positive integers $a$, $(a,q)=1$ with $M(a) \le M$. 
\end{problem}

    \bp

Let us make one last remark.
In papers 
\cite{chang2011partial}, 
\cite[Corollary 5]{Larcher}, 
\cite{Mosh_A+B}, 
\cite{MU_Larcher}, \cite{sh_representation_SP}
the authors study the following question.
Suppose that $q=p$ is a prime number, and the maximum in \eqref{def:M(a)} and the sum in \eqref{def:S(a)} are taken over a sufficiently large multiplicative subgroup of $\G$.
Is it possible to obtain analogs of \eqref{f:Korobov_log}, \eqref{f:main_M}, \eqref{f:Larcher_old} and \eqref{f:main_M_new} under the additional assumption that $a\in \G \subseteq \F_p \setminus \{0\}$?
Using exactly the same argument as in \cite[Theorem 5]{sh_representation_SP}, one can find $a\in \G$ such that Theorem \ref{t:main_new} (and similarly for Theorems  \ref{t:Larcher} and \ref{t:Zaremba}, the parameter $M$ can vary) hold for $a\in \G$, provided that
\[
|\G| \gg p^{1-\frac{\kappa_*}{M}} 
\gg p \cdot \exp (-O(\sqrt{\log p})) \,, 
\]
where $\kappa_* >0$ is an absolute constant.

\subsection{The proof of Theorem \ref{t:Zaremba}}
\label{ssec:Zaremba}

Now we are ready to prove Theorem \ref{t:Zaremba} and we use the notation and the argument of the previous subsection. 
We do not use the exhaustion procedure as we did in the proofs of Theorems \ref{t:Larcher}, \ref{t:main_new}. 
Instead of this we consider the set $A = A_M := Z_M (t) \cap Z_{M_*} (tH)$.
As before we assume that condition  \eqref{cond:eps} takes place.
Similarly, put $N_* = N/H^2 = N^{1/10}$ and defining $\eps_*$ from $q^{1/2-\eps_*} = tH$, we assume that 
\begin{equation}\label{cond:eps*}
    \eps_* \gg \frac{1}{M_*} \,.
\end{equation}
In other words, $\eps_* = \eps/10$ and our basic condition  \eqref{cond:eps} implies \eqref{cond:eps*} if one puts $M_* \approx 10 M$ (more precisely, we need $1-w_{M} = 10 (1-w_{M_*})$), consult Theorem \ref{t:Hensley_HD}.
Further we compute the size of $A$ using Lemma \ref{l:AD} and arguments as in  \eqref{f:A_M/M_*}.
Namely, using \eqref{cond:I&M_h}, we get 
\[
    |A| \gg 
        M_*^{-1} t^{2w_M} N^{w_{M_*}} (N/H^2)^{1-w_{M_*}} = M^{-1}_* t^{2w_M} N H^{2(w_{M_*}-1)} 
        =
        M^{-1}_* q^{w_M} N^{1-w_M} H^{2(w_{M_*}-1)}  \,,
\]
and hence by Corollary \ref{c:A_A*_asymptotic} and Lemma \ref{l:coprimality} applied to the set $A$, we obtain 
\[
    |A\cap A^{-1}| \ge \frac{|A|^2 \_phi (q)}{q^2} - O(|A| N^{-\kappa}_*) \ge \frac{|A|^2\_phi (q)}{2q^2} \,,
\]
    provided 
\[
    M^{-1}_* q^{w_M} N^{1-w_M} H^{2(w_{M_*}-1)} N^{\kappa}_* 
    \gg 
    M^{-1} q^{w_M} N^{w_{M_*}-w_M+(1-w_{M_*})/10}  \cdot  N^{\kappa/10} 
\]
\[
    \ge 
        M^{-1} q^{w_M} N^{\kappa/10}
    \gg q^{1+o(1)}
\]
and this is follows from \eqref{cond:eps*}. 
Notice that if $M \to \infty$, then  the final condition on $\eps$ is a strengthening of \eqref{cond:eps_upper}, namely,
\begin{equation}\label{cond:eps_upper_final}
		\eps< \frac{1}{18} - o_M (1) \,.
\end{equation}
Thus, there exists a subset of $A\cap A^{-1}$ of positive density such that all partial quotients of its elements 
are bounded by
\begin{equation}\label{f:M_upper_bound}
    \max\{ O(M_*), \tilde{M}\} \sim \max\{ 10 (\kappa \eps)^{-1}, 100 \d^{-60} \} = O(1) \,,
\end{equation}
where $\eps$ satisfies \eqref{cond:eps_upper_final} and $\eps \sim \kappa$ in the case of a general composite $q$. 
If $\delta \to 1$ (of course $\d$ should be less than $1-O(N^{-\kappa})$, see Lemma \ref{l:dense_subint}), then the contribution of the second term in the maximum from \eqref{f:M_upper_bound} is negligible (but the number of $a$ in \eqref{f:Zaremba} decreases) and we see that indeed  the main losses arise  thanks to the quantity $\kappa$ and hence from the results concerning the growth in the modular group.
This completes the proof of the theorem. 
$\hfill\Box$

\section{Appendix}
\label{sec:appendix}

In this section, we give a rough upper bound for the absolute constant $\mathcal{M}$ from  Theorem \ref{t:Zaremba} for the main case where the denominator $q$ is  prime. 
Moreover, we want to find only a single $a$  and do not address the question of lower bounds for the numerators in Zaremba's problem.
For prime denominators, the situation is simpler: the condition \eqref{cond:gcd} is satisfied automatically, and, furthermore, we do not make $\d$--Assumption assuming that {\it each} $a\in A\cap A^{-1}$ has a $\tilde{M}$--critical denominator. 
We claim that for sufficiently large $q=p$ one can take 
\begin{equation}\label{f:M_upper_bound_p}
    \mathcal{M} =  2^{2000} \,.
\end{equation}
The proof of \eqref{f:M_upper_bound_p} relies on arguments from  \cite{MMS_popular} and \cite{MMS_Korobov}, incorporating specific elements from our other papers.
First of all, it is necessary to transform the approximation $M_* \approx 10 M$ from Subsection \ref{ssec:Zaremba} into a rigorous statement (recall that we need to solve the inequality $1-w_{M} \ge 10 (1-w_{M_*})$).
To this end we need \cite[Theorem VIII]{Kurzweil_w_M} 
(a similar bound for any $M\ge 2$ is contained in \cite[Theorem 1]{hensley1989distribution}).

\begin{theorem}
    For $M\ge 1000$ one has
\[
    1-\frac{0.99}{M}  \le w_M \le 1-\frac{1}{4M} \,.
\]
\label{t:w_M_bounds}
\end{theorem}

Applying Theorem \ref{t:w_M_bounds}, we see that $1-w_{M} \ge  10 (1-w_{M_*})$, provided $M_* \ge 40 M$ and $M\ge 1000$.

Now we are ready move to the main part of our calculations. 
Since we do not make $\d$--Assumption the constant $\tilde{M}$ can be taken $\tilde{M} = 200$, say, and in view of the second maximum in \eqref{f:M_upper_bound}, we see that it is enough to compute the constant $\kappa$ from Lemmas \ref{l:T-action}, \ref{l:BG_new}. 
To this end we follow the notation \cite[Section 2]{MMS_Korobov}.
We consider the set  of matrices 
\begin{equation}\label{def:G} 
G=\left\{\left(\begin{array}{cc}
1 & -2 j \\
2 j & 1-4 j^{2}
\end{array}\right): 1 \leq j \leq N\right\} \subset \SL_2 (\Z/p\Z) \,,
\end{equation}
and put 
$\tau =1/4$, 
$m = \tau \log_N p$, $K(G^m) = p^{\tau/6}$ 
(see \cite[Lemmas 12, 21]{MMS_popular} and \cite[Page 7]{MMS_Korobov}). 
In this terms 
the final formula for the constant $\kappa$ is (see \cite[Page 9]{MMS_Korobov})
\[
    \log N \cdot \kappa = (6m)^{-1} \tau \d \log p = 
    6^{-1} \d \log N \,, 
\] 
where $\d = c/2^{k+4}$, $k = \lceil c^{-1} \rceil +1$ and $c>0$ is an absolute constant.

It remains to estimate the constant $c$ from below and 
in the proof of the required lower bound for $c$ 
we used a version of the Balog--Szemer\'edi--Gowers theorem \cite{TV} that belongs to Murphy \cite{Murphy_group_action}. 
The constants $C_1$, $C_2$ from Theorem \ref{t:BSzG_Murphy}  can be obtained if one combines Theorem 2.29 and Lemma 2.13 (which is \cite[Proposition 4.5]{Tao_non-commutative}) of \cite{TV}, see details in \cite{Murphy_group_action}.  

\begin{theorem} 
    Let $\Gr$ be a group, $A\subseteq \Gr$ and $\E(A) \ge |A|^3/K$. 
    Then there is $a\in A$ and  a set $A_* \subseteq a^{-1}A$ such that $|A_*| \gg |A|/K^{C_1}$ and  $|A^3_*| \ll K^{C_2} |A_*|$.
    One can take $C_1=9$ and $C_2 = 32$. 
\label{t:BSzG_Murphy}
\end{theorem}

Also, we need \cite[Theorem 5]{RS_SL2+}. 

\begin{lemma}
\label{l:Helfgott}
    Let $p$ be a prime number and $A\subseteq \SL_2 (\F_p)$ be a generating set. 
    Then 
\[
    |A^3| \gg \min\{ |\SL_2 (\F_p)|, |A|^{1+c_H}\} \,,
\]
    where $c_H=1/20$. 
\end{lemma}

Then,  following the argument of the proof of \cite[Theorem 49]{sh_non_survey} in the case of $\SL(\F_p)$, we find 
\[
    c \ge \min\{1/3,  (8C_2)^{-1}, \tau (4C_2)^{-1}, 0.5 c_H (C_1+ C_2)^{-1} \} = 1/1640 \,.
\] 


It implies that  $k = \lceil c^{-1} \rceil +1 = 1641$.
So, $\kappa \ge 2^{-1656}$, and in view of the bound \eqref{f:M_upper_bound} our desired estimate \eqref{f:M_upper_bound_p} follows.


\bibliographystyle{abbrv}

\bibliography{bibliography}{}

\begin{thebibliography}{10}

\bibitem{ABH_Ruk+}
C.~Aistleitner, B.~Borda, and M.~Hauke.
\newblock {On the distribution of partial quotients of reduced fractions with fixed denominator}.
\newblock {\em Transactions of the American Mathematical Society}, 377(02):1371--1408, 2024.

\bibitem{BD_Iab}
J.~Bourgain and S.~Dyatlov.
\newblock {Fourier dimension and spectral gaps for hyperbolic surfaces}.
\newblock {\em Geometric and Functional Analysis}, 27:744--771, 2017.

\bibitem{BD_AD_def}
J.~Bourgain and S.~Dyatlov.
\newblock {Spectral gaps without the pressure condition}.
\newblock {\em Annals of Mathematics}, 187(3):825--867, 2018.

\bibitem{BG_p^n}
J.~Bourgain and A.~Gamburd.
\newblock Expansion and random walks in $\mathrm{SL}_d (\mathbb{Z}/p^n \mathbb{Z})$: I.
\newblock {\em Journal of the European Mathematical Society}, 10(4):987--1011, 2008.

\bibitem{BG}
J.~Bourgain and A.~Gamburd.
\newblock {Uniform expansion bounds for Cayley graphs of $\mathrm{SL}_2 (\mathbb{F}_p)$}.
\newblock {\em Annals of Mathematics}, pages 625--642, 2008.

\bibitem{BGS_affine}
J.~Bourgain, A.~Gamburd, and P.~Sarnak.
\newblock {Affine linear sieve, expanders, and sum-product}.
\newblock {\em Inventiones mathematicae}, 179(3):559--644, 2010.

\bibitem{bourgain2011zarembas}
J.~Bourgain and A.~Kontorovich.
\newblock {On Zaremba's conjecture}.
\newblock {\em Comptes Rendus Mathematique}, 349(9-10):493--495, 2011.

\bibitem{BK_Zaremba}
J.~Bourgain and A.~Kontorovich.
\newblock {On Zaremba's conjecture}.
\newblock {\em Annals of Mathematics}, pages 137--196, 2014.

\bibitem{CSP_Zaremba}
S.~H. Chan, S.~Heilman, and G.~Panova.
\newblock {Independent Sets and Continued Fractions}.
\newblock {\em arXiv preprint arXiv:2604.19094}, 2026.

\bibitem{CKP_trees_Zaremba}
S.~H. Chan, A.~Kontorovich, and I.~Pak.
\newblock {Spanning trees and continued fractions}.
\newblock {\em arXiv preprint arXiv:2411.18782}, 2024.

\bibitem{CKP_planar_CF2}
S.~H. Chan, A.~Kontorovich, and I.~Pak.
\newblock {Effective resistance in planar graphs and continued fractions}.
\newblock {\em arXiv preprint arXiv:2505.19168}, 2025.

\bibitem{CP_linear}
S.~H. Chan and I.~Pak.
\newblock Linear extensions and continued fractions.
\newblock {\em European Journal of Combinatorics}, 122:104018, 2024.

\bibitem{CPS_loglog}
S.~H. Chan, I.~Pak, and I.~Shkredov.
\newblock {\em preprint}, 2025.

\bibitem{chang2011partial}
M.-C. Chang.
\newblock {Partial quotients and equidistribution}.
\newblock {\em Comptes Rendus Mathematique}, 349(13-14):713--718, 2011.

\bibitem{Dyatlov_AD}
S.~Dyatlov.
\newblock {An introduction to fractal uncertainty principle}.
\newblock {\em Journal of Mathematical Physics}, 60(8), 2019.

\bibitem{DZ_AD}
S.~Dyatlov and J.~Zahl.
\newblock {Spectral gaps, additive energy, and a fractal uncertainty principle}.
\newblock {\em Geometric and Functional Analysis}, 26(4):1011--1094, 2016.

\bibitem{Fraser_almostAP}
J.~M. Fraser and H.~Yu.
\newblock {Arithmetic patches, weak tangents, and dimension}.
\newblock {\em Bulletin of the London Mathematical Society}, 50(1):85--95, 2018.

\bibitem{Helfgott_growth}
H.~A. Helfgott.
\newblock {Growth and generation in $\SL_2(\Z/p\Z)$}.
\newblock {\em Annals of Mathematics}, pages 601--623, 2008.

\bibitem{hensley1989distribution}
D.~Hensley.
\newblock {The distribution of badly approximable numbers and continuants with bounded digits}.
\newblock {\em Th{\'e}orie des nombres (Quebec, PQ, 1987)}, pages 371--385, 1989.

\bibitem{hensley1992continued}
D.~Hensley.
\newblock {Continued fraction Cantor sets, Hausdorff dimension, and functional analysis}.
\newblock {\em Journal of number theory}, 40(3):336--358, 1992.

\bibitem{hensley_SL2}
D.~Hensley.
\newblock {The distribution mod $n$ of fractions with bounded partial quotients}.
\newblock {\em Pacific Journal of Mathematics}, 166(1):43--54, 1994.

\bibitem{hensley1996}
D.~Hensley.
\newblock {A polynomial time algorithm for the Hausdorff dimension of continued fraction Cantor sets}.
\newblock {\em Journal of Number Theory}, 58(1):9--45, 1996.

\bibitem{Hinchin}
A.~Y. Hinchin.
\newblock {Continuous Fractions}, 1967.

\bibitem{Hlawka}
E.~Hlawka.
\newblock {Funktionen von beschr{\"a}nkter variatiou in der theorie der gleichverteilung}.
\newblock {\em Annali di Matematica Pura ed Applicata}, 54(1):325--333, 1961.

\bibitem{KanIV}
I.~D. Kan.
\newblock {A strengthening of a theorem of Bourgain and Kontorovich IV}.
\newblock {\em Izvestiya: Mathematics}, 80(6):1094, 2016.

\bibitem{Kan_3thms}
I.~D. Kan.
\newblock {A strengthening of the Bourgain-Kontorovich method: three new theorems}.
\newblock {\em Sbornik: Mathematics}, 212(7):921--964, 2021.

\bibitem{Kesten}
H.~Kesten.
\newblock {Symmetric random walks on groups}.
\newblock {\em Transactions of the American Mathematical Society}, 92(2):336--354, 1959.

\bibitem{Koksma}
J.~F. Koksma.
\newblock {A general theorem from the theory of uniform distribution modulo $1$}.
\newblock {\em Mathematica B (Zutphen)}, 11:7--11, 1942.

\bibitem{Kontorovich_survey}
A.~Kontorovich.
\newblock {From Apollonius to Zaremba: local-global phenomena in thin orbits}.
\newblock {\em Bulletin of the American Mathematical Society}, 50(2):187--228, 2013.

\bibitem{Korobov_optimal_I}
N.~M. Korobov.
\newblock {The evaluation of multiple integrals by method of optimal coefficients}.
\newblock {\em Vestnik Moskovskogo universiteta}, (4):19--25, 1959.

\bibitem{Korobov_Zaremba}
N.~M. Korobov.
\newblock {Properties and calculation of optimal coefficients}.
\newblock {\em Soviet Math. Dokl.}, (1):696--700, 1960.

\bibitem{Korobov_optimal_II}
N.~M. Korobov.
\newblock {Properties and calculation of optimal coefficients}.
\newblock {\em Dokl. Akad. Nauk SSSR}, 132(5):1009--1012, 1960.

\bibitem{Korobov_book}
N.~M. Korobov.
\newblock {Number-theoretic methods in numerical analysis}.
\newblock {\em Fizmatgis, Moscow}, 37, 1963.

\bibitem{KS_linear}
N.~Kravitz and A.~Sah.
\newblock {Linear extension numbers of n-element posets}.
\newblock {\em Order}, 38:49--66, 2021.

\bibitem{KN_book}
L.~Kuipers and H.~Niederreiter.
\newblock {\em {Uniform distribution of sequences}}.
\newblock Courier Corporation, 2012.

\bibitem{Kurzweil_w_M}
J.~Kurzweil.
\newblock {A contribution to the metric theory of diophantine approximations}.
\newblock {\em Czechoslovak Mathematical Journal}, 1(3):149--178, 1951.

\bibitem{Larcher}
G.~Larcher.
\newblock {On the distribution of sequences connected with good lattice points}.
\newblock {\em Monatshefte f{\"u}r Mathematik}, 101:135--150, 1986.

\bibitem{Larcher_survey}
G.~Larcher.
\newblock {Discrepancy estimates for sequences: new results and open problems}.
\newblock {\em Uniform Distribution and Quasi-Monte Carlo Methods: Discrepancy, Integration and Applications}, 15:171, 2014.

\bibitem{Mcmullen}
C.~T. McMullen.
\newblock {Uniformly Diophantine numbers in a fixed real quadratic field}.
\newblock {\em Compositio Mathematica}, 145(4):827--844, 2009.

\bibitem{Mercat_Z}
P.~Mercat.
\newblock {Construction de fractions continues p{\'e}riodiques uniform{\'e}ment born{\'e}es}.
\newblock {\em Journal de th{\'e}orie des nombres de Bordeaux}, 25(1):111--146, 2013.

\bibitem{PC1}
N.~Moshchevitin and I.~Kan.
\newblock Personal communications.

\bibitem{MMS_popular}
N.~Moshchevitin, B.~Murphy, and I.~Shkredov.
\newblock {Popular products and continued fractions}.
\newblock {\em Israel J. Math.}, 238(2):807--835, 2020.

\bibitem{MMS_Korobov}
N.~Moshchevitin, B.~Murphy, and I.~Shkredov.
\newblock {On Korobov bound concerning Zaremba’s conjecture}.
\newblock {\em IMRN, accepted}, pages 1--19, 2026.

\bibitem{Mosh_A+B}
N.~G. Moshchevitin.
\newblock {Sets of the form $A+B$ and finite continued fractions}.
\newblock {\em Sbornik: Mathematics}, 198(4):537--557, 2007.

\bibitem{Mosh_survey}
N.~G. Moshchevitin.
\newblock {On some open problems in Diophantine approximation}.
\newblock {\em arXiv preprint arXiv:1202.4539}, 2012.

\bibitem{MS_Zaremba_mod}
N.~G. Moshchevitin and I.~D. Shkredov.
\newblock On a modular form of {Z}aremba's conjecture.
\newblock {\em Pacific J. Math.}, 309(1):195--211, 2020.

\bibitem{MU_Larcher}
N.~G. Moshchevitin and D.~M. Ushanov.
\newblock {On Larcher's theorem concerning good lattice points and multiplicative subgroups modulo $p$}.
\newblock {\em Uniform Distribution Theory}, 5(1):45--52, 2010.

\bibitem{Murphy_group_action}
B.~Murphy.
\newblock {Group action combinatorics}.
\newblock {\em arXiv preprint arXiv:1907.13569}, 2019.

\bibitem{Niederreiter_dyadic}
H.~Niederreiter.
\newblock Dyadic fractions with small partial quotients.
\newblock {\em Monatshefte f{\"u}r Mathematik}, 101:309--315, 1986.

\bibitem{RS_SL2+}
M.~Rudnev and I.~D. Shkredov.
\newblock {On the growth rate in $\SL_2 (\F_p)$, the affine group and sum-product type implications}.
\newblock {\em Mathematika}, 68(3):738--783, 2022.

\bibitem{Ruk}
M.~Rukavishnikova.
\newblock {Probabilistic bound for the sum of partial quotients of fractions with a fixed denominator}.
\newblock {\em Chebyshevskii sbornik}, 7:113--121, 2006.

\bibitem{Schmidt}
W.~Schmidt.
\newblock {Irregularities of distribution, VII}.
\newblock {\em Acta Arithmetica}, 21(1):45--50, 1972.

\bibitem{sh_non_survey}
I.~D. Shkredov.
\newblock {Non-commutative methods in additive combinatorics and number theory}.
\newblock {\em Russian Mathematical Surveys}, 76(6):1065, 2021.

\bibitem{s_Chevalley}
I.~D. Shkredov.
\newblock {Growth in Chevalley groups relatively to parabolic subgroups and some applications}.
\newblock {\em Revista Mathematica Iberoamericana}, 38(6):1945--1973, 2022.

\bibitem{sh_BG}
I.~D. Shkredov.
\newblock {On a girth--free variant of the Bourgain--Gamburd machine}.
\newblock {\em Finite Fields and Their Applications}, 90:102225, 2023.

\bibitem{sh_representation_SP}
I.~D. Shkredov.
\newblock {Some applications of representation theory to the sum-product phenomenon}.
\newblock {\em arXiv:2307.03156}, 2023.

\bibitem{Shulga_Zaremba}
N.~Shulga.
\newblock {Radical bound for Zaremba's conjecture}.
\newblock {\em Bulletin of the London Mathematical Society}, 56(8):2615--2624, 2024.

\bibitem{Tang_Zhang_super_app}
J.~Tang and X.~Zhang.
\newblock {Super approximation for $SL_2 \times SL_2$ and $ASL_2$}.
\newblock {\em arXiv preprint arXiv:2308.09982}, 2023.

\bibitem{Tang_Zhang_SP}
J.~Tang and X.~Zhang.
\newblock {Sum-product in quotients of rings of algebraic integers}.
\newblock {\em Journal d'Analyse Math{\'e}matique}, 2025.

\bibitem{Tao_non-commutative}
T.~Tao.
\newblock {Product set estimates for non-commutative groups}.
\newblock {\em Combinatorica}, 28(5):547--594, 2008.

\bibitem{TV}
T.~Tao and V.~Vu.
\newblock {\em Additive combinatorics}, volume 105 of {\em Cambridge Studies in Advanced Mathematics}.
\newblock Cambridge University Press, Cambridge, 2006.

\bibitem{zaremba1966}
S.~K. Zaremba.
\newblock {Good lattice points, discrepancy, and numerical integration}.
\newblock {\em Annali di matematica pura ed applicata}, 73:293--317, 1966.

\bibitem{zaremba1972methode}
S.~K. Zaremba.
\newblock {La m{\'e}thode des “bons treillis” pour le calcul des int{\'e}grales multiples}.
\newblock In {\em Applications of number theory to numerical analysis}, pages 39--119. Elsevier, 1972.

\bibitem{Zhang_Zaremba_q}
X.~Zhang.
\newblock {Expansion in $SL_2(\Z/q\Z)$ and Zaremba's conjecture}.
\newblock {\em arXiv preprint arXiv:2605.02518}, 2026.

\end{thebibliography}

\end{document}